\documentclass[a4paper, 12pt, twoside]{article}

\title{A variational representation for $G$-Brownian functionals}
\author{Emi Osuka\thanks{Mathematical Institute, 
Tohoku University, Aoba-ku, Sendai 980-8578, Japan.}}
\date{\empty}

\usepackage{amsmath,mathtools,amsthm,enumerate}
\usepackage{amssymb,mathrsfs,latexsym}%%%%%fonts
\usepackage{fancyhdr}
\allowdisplaybreaks
\numberwithin{equation}{section}
\mathtoolsset{showonlyrefs=true}

\theoremstyle{plain}
	\newtheorem{thm}{Theorem}[section]
	\newtheorem{prop}[thm]{Proposition}
	\newtheorem{lem}[thm]{Lemma}

\theoremstyle{definition}
	
\theoremstyle{remark}
	\newtheorem{rem}[thm]{Remark}
	
\newcommand	\tref	{Theorem~\ref}
\newcommand	\pref	{Proposition~\ref}
\newcommand	\lref	{Lemma~\ref}
\newcommand	\cref	{Corollary~\ref}

\newcommand	\rref	{Remark~\ref}

\newcommand	\A	{\mathbb{A}}		
		
\newcommand	\E	{\mathbb{E}}		
		\newcommand	\bH	{\mathbb{H}}

		\newcommand	\N	{\mathbb{N}}
		
		\newcommand	\R	{\mathbb{R}}

		\newcommand	\X	{\mathbb{X}}

\newcommand	\calA	{\mathcal{A}}		\newcommand	\calB	{\mathcal{B}}
		
\newcommand	\calE		{\mathcal{E}}		\newcommand	\calF		{\mathcal{F}}
		\newcommand	\calH		{\mathcal{H}}
		
		\newcommand	\calL		{\mathcal{L}}
		\newcommand	\calN		{\mathcal{N}}
		\newcommand	\calP		{\mathcal{P}}

	\newcommand	\calX	{\mathcal{X}}

\newcommand	\qv	[1]	{\langle #1 \rangle}
\newcommand	\fp	[2]	{\frac{\partial #1}{\partial #2}}
\newcommand	\bt		{\textbf}
\newcommand	\bm		{\boldsymbol}

\newcommand	\ddd		{, \dots ,}
\newcommand	\n			{\text}
\newcommand	\one		{\normalfont{\mbox{1}\hspace{-0.25em}\mbox{l}}}

\newcommand	\qd		{\quad}
\newcommand	\qn		{\quad\text}

\newcommand	\tr		{\mathrm{tr}}

\newcommand	\ve		{\varepsilon}
\newcommand	\vp		{\varphi}

\renewcommand	\le	{\leqslant}
\renewcommand	\ge	{\geqslant}
\renewcommand	\hat	{\widehat}

\newcommand	\ol		{\bar}

\newcommand	\OL		{\overline}

\newcommand	{\bLip}	{C_{b,Lip}}
\newcommand	{\sbl}	{\mathcal{S}_{b,Lip}}
\newcommand	{\cA}	{\calA_{0,1}^{\Theta}}
\newcommand	{\lip}		{\mathrm{Lip}}%% Lipschitz constant
\newcommand	{\MG}	{M^2_G (0,1)}
\newcommand	{\MGR}	{M^2_G (0,1 ; \R ^d)}

\begin{document}

\maketitle

\begin{abstract}
The purpose of this paper is to establish a variational representation
\begin{align}
	\log \E \left[ e^{f(B)} \right]
	=
	\sup_h
	\E 
	\left[ 
	f \left( B + \int _0^{ \cdot } d \qv{B}_s \, h_s \right) 
	-\frac{1}{2} \int _0^1 h_s \cdot \left( d \qv{B}_s \, h_s \right) 
	\right]
\end{align}
for functionals of the $d$-dimensional $G$-Brownian motion $B$.
Here $\E$ is a sublinear expectation called $G$-expectation,
$f$ is any bounded function in the domain of $\E$ 
mapping $C ([0,1] ; \R^d)$ to $\R$,
the integrals are taken with respect to the quadratic variation of $B$,
and the supremum runs over all $h$'s for which these integrals are well-defined.
As an application, we give another proof of the results obtained by Gao-Jiang (2010),
large deviations for $G$-Brownian motion.

\end{abstract}

\footnote{e-mail: sa9m06@math.tohoku.ac.jp}
\footnote{phone: +81 227956401, fax: +81 227956400}
\footnote{\bt{Keywords}: 
$G$-Brownian motion, Variational representation, Large deviations, Girsanov's formula}

%%%%%%%%%%%%%%%%%%%%%%%%%%%%%%%%%%%%%%%%%%%%%%%%%%%%%%%%%%%%%%%%%%%%%
%%%%%%%%%%%%%%%%%%%%%%%%%%%%%%%%%%%%%%%%%%%%%%%%%%%%%%%%%%%%%%%%%%%%%
%%%%%%%%%%%%%%%%%%%%%%%%%%%%%%%%%%%%%%%%%%%%%%%%%%%%%%%%%%%%%%%%%%%%%
%%%%%%%%%%%%%%%%%%%%%%%%%%%%%%%%%%%%%%%%%%%%%%%%%%%%%%%%%%%%%%%%%%%%%
%%%%%%%%%%%%%%%%%%%%%%%%%%%%%%%%%%%%%%%%%%%%%%%%%%%%%%%%%%%%%%%%%%%%%
\section{Introduction}\label{;intro}

This paper is concerned with $G$-Brownian motion introduced by S.~Peng.
$G$-Brownian motion can be regarded 
as a Brownian motion with an uncertain variance process.
One of its features is that, 
while the classical Brownian motion is defined on a probability space, 
$G$-Brownian motion is defined 
on a sublinear expectation space $(\Omega , \calH , \E)$.
Here $\Omega$ is a given set and $\calH$ 
is a vector lattice of real-valued functions on $\Omega$ 
containing $1$, which is the domain of a sublinear expectation $\E$.
Peng \cite{Peng;06,Peng;08b} constructed 
a sublinear expectation space on which the canonical 
process of the space $\Omega = C ([0,1] ; \R^d)$ 
of continuous paths starting from 0 
becomes a $G$-Brownian motion.
The sublinear expectation in this space is called $G$-expectation.
Also defined in \cite{Peng;06,Peng;08b} 
were the quadratic variation process of $G$-Brownian motion, 
and stochastic integrals with respect to $G$-Brownian motion 
and its quadratic variation for a certain class of stochastic processes.
It is known that any sublinear expectation can be represented 
as a supremum of linear expectations, referred to as an upper expectation.
Recently, L.~Denis, M.~Hu and S.~Peng gave 
a concrete upper expectation representation 
for $G$-expectation in \cite{Denis;10}.
Through the upper expectation, a related capacity is defined, 
and it plays a similar role to a probability measure in the classical stochastic analysis.
For instance, Gao-Jiang \cite{Gao;10} formulated and proved large deviation principles 
for $G$-Brownian motion under this capacity.

In this paper, we establish a variational representation 
for functionals of $G$-Brownian motion:
\begin{align}
	\log \E \left[ e^{f(B)} \right]
	=
	\sup _{h \in ( \MG )^d}
	\E 
	\left[ 
	f \left( B + \int _0^{\cdot} d \qv{B}_s \, h_s \right) 
	- \frac{1}{2} \int _0^1 h_s \cdot \left( d \qv{B}_s \, h_s \right) 
	\right] . 
	\label{eq:intro}
\end{align}
Here $\E$ is $G$-expectation, 
$B$ is the $d$-dimensional $G$-Brownian motion,
$f$ is any bounded function in the domain of $G$-expectation 
that maps $C ([0,1] ; \R^d)$ to $\R$, 
the integrals are taken with respect to 
the quadratic variation $\qv{B}$ of $G$-Brownian motion, 
and the supremum runs over all $\R^d$-valued processes $h$ 
for which these integrals are well-defined. 
Precise definitions will be seen in Section~\ref{;prem}.

One of our motivations for this representation 
comes from large deviation principles for $G$-Brownian motion.
It is well known that, on a probability space, 
the large deviation principle for a given family 
of random variables is equivalent to its Laplace principle. 
In \cite{Boue;98}, M.~Bou\'e and P.~Dupuis 
established a variational representation for functionals 
of Brownian motion and showed its usefulness 
in the derivation of Laplace principles when the family 
of concern consists of functionals of Brownian motion.
Our variational representation \eqref{eq:intro} 
has the same application in the framework of $G$-expectation space; 
indeed, it is also true that the Laplace principle formulated under $G$-expectation 
is equivalent to the large deviation principle 
formulated under the capacity, and the representation \eqref{eq:intro} 
can be used to derive Laplace principles for families of random variables 
given as functionals of $G$-Brownian motion.
As an illustration, we prove the Laplace principles for the families 
$\{ \sqrt{\ve} B ; \ve>0 \}$ and $\{ ( \sqrt{\ve} B , \qv{B} ) ; \ve >0 \}$.
Large deviations for these families were originally obtained by Gao-Jiang \cite{Gao;10};
they employed a discretization technique.
Our variational representation gives another proof.

The proof of the representation \eqref{eq:intro} 
is split into the derivations of the lower and upper bounds. 
By virtue of  approximating method we employ, 
proofs of these bounds are reduced to showing their validity 
for a particular class of functions $f$, 
namely the class of bounded Lipschitz cylinder functions.
To obtain the lower bound, Girsanov's formula 
for $G$-Brownian motion in \cite{Osuka;11} 
allows us to use a similar argument to that in Bou\'e-Dupuis \cite{Boue;98}.
The proof of the upper bound is in the same spirit as Zhang \cite{Zhang;09}, 
which extended the representation of Bou\'e-Dupuis 
to the framework of an abstract Wiener space 
as simplifying the proof of the upper bound by using the Clark-Ocone formula; 
we use a type of the Clark-Ocone formula under $G$-expectation (\lref{lem:u2}) 
to prove the upper bound. 
Prior to the proof of the representation \eqref{eq:intro}, 
the well-definedness of the right-hand side has also to be verified,
that is, it is needed to show that for any bounded function $f$ 
in the domain of $G$-expectation, 
functionals of the form $f ( B + \int_0^{\cdot} d \qv{B}_s \, h_s )$ with 
$h$ as described above are again in the domain.
A key is to establish an absolute continuity between $B$ and 
$B + \int_0^{\cdot} d \qv{B}_s \, h_s$ under the capacity (\pref{;abs}), 
which is done by using relative entropy estimates as given in Bou\'e-Dupuis 
\cite{Boue;98} and also by using Girsanov's formula for $G$-Brownian motion.

We give an outline of the paper.
In Section~\ref{;prem}, we introduce necessary notions 
and related results as preliminaries:
the construction of $G$-expectation, 
stochastic integrals for $G$-Brownian motion,
the upper expectation for $G$-expectation due to Denis-Hu-Peng \cite{Denis;10},
and Girsanov's formula for $G$-Brownian motion obtained by \cite{Osuka;11}.
Main results of this paper are stated and proved in Section~\ref{;pr2};
we verify the well-definedness of the right-hand side 
of \eqref{eq:intro} in Subsection~\ref{;pr1},
and prove the representation \eqref{eq:intro} 
in Subsections~\ref{;prflb} and \ref{;prfub}. 
In Section~\ref{;ld}, we derive large deviation principles for $G$-Brownian motion 
as an application of our representation. 
In Section~\ref{s:rem}, we show an absolute continuity relationship between $B$ and 
$B + \int_0^{\cdot} d \qv{B}_s \, h_s$ under the capacity.

Throughout this paper, for a probability measure $P$, 
$E_P$ denotes the expectation with respect to $P$.
For a real-valued function $f$ on any metric space 
$( \X , d )$, we denote by $\lip (f)$ the Lipschitz constant of $f$: 
\begin{align}
	\lip (f)
	:=
	\sup _{
	\begin{subarray}{c}
		x , y \in \X\\
		x \neq y
	\end{subarray}
	}
	\frac{ |f(x)-f(y)| }{ d(x,y) }. 
\end{align}
Other notation will be introduced as needed.

%%%%%%%%%%%%%%%%%%%%%%%%%%%%%%%%%%%%%%%%%%%%%%%%%%%%%%%%%%%%%%%%%%%%%
%%%%%%%%%%%%%%%%%%%%%%%%%%%%%%%%%%%%%%%%%%%%%%%%%%%%%%%%%%%%%%%%%%%%%
%%%%%%%%%%%%%%%%%%%%%%%%%%%%%%%%%%%%%%%%%%%%%%%%%%%%%%%%%%%%%%%%%%%%%
%%%%%%%%%%%%%%%%%%%%%%%%%%%%%%%%%%%%%%%%%%%%%%%%%%%%%%%%%%%%%%%%%%%%%
%%%%%%%%%%%%%%%%%%%%%%%%%%%%%%%%%%%%%%%%%%%%%%%%%%%%%%%%%%%%%%%%%%%%%
\section{$\bm{G}$-Brownian motion and related stochastic analysis}\label{;prem}

In this section, we briefly recall from \cite{Peng;06,Peng;08b,Denis;10} 
some notions and related results about 
$G$-Brownian motion and $G$-expectation space. As preparing some 
necessities such as the notions of $G$-stochastic integrals and 
$G$-martingales, we then introduce Girsanov's formula for multidimensional 
$G$-Brownian motion established in \cite{Osuka;11}. 

%%%%%%%%%%%%%%%%%%%%%%%%%%%%%%%%%%%%%%%%%%%%%%%%%%%%%%%%%%%%%%%%%%%%%
%%%%%%%%%%%%%%%%%%%%%%%%%%%%%%%%%%%%%%%%%%%%%%%%%%%%%%%%%%%%%%%%%%%%%
\subsection{$\bm{G}$-expectation space and the related capacity}

Let $\Omega $ be the set of $\R^d$-valued continuous functions 
$\omega :[0,1]\to \R ^{d}$ with $\omega_0 =0$, equipped with the distance 
\begin{align}\label{eq:rho}
	\rho (\omega^1 , \omega^2) := \sup _{0\le t\le 1} 
	| \omega^1_t - \omega^2_t | , \quad \omega ^{1}, \omega ^{2}\in \Omega . 
\end{align}
For each $t \in [0,1]$, we also set 
$\Omega_t := \{ \omega _{\cdot \wedge t} : \omega \in \Omega \}$. 
We denote by $\calB ( \Omega )$ (resp.~by $\calB ( \Omega_t )$)
the associated Borel $\sigma$-algebra of $\Omega$ (resp.~of $\Omega_t$).
In the sequel we denote by $B=\{ B_t ; 0 \le t \le 1 \}$ the canonical process in
$\Omega$: $B_t(\omega ):=\omega _{t},\,0\le t\le 1, \omega \in \Omega $. 
For each $t\in [0,1]$, let $\bLip (\Omega _{t})$ be 
the set of bounded Lipschitzian cylinder functionals on $\Omega _{t}$:
\begin{align*}
	\bLip ( \Omega_t )
	:= \left\{ 
	\vp ( B_{t_1} \ddd B_{t_n}) : \, n \in \N ,~ 
	t_1 \ddd t_n \in [0,t] ,~ \vp \in \bLip ( (\R^d)^ n ) 
	\right\} ;
\end{align*}
when $t=1$, we simply write $\bLip (\Omega )$.
Here and below, $\bLip (\R ^{m})$ denotes 
the set of bounded Lipschitz functions on $\R ^{m}$.
Let $\R ^{d\times d}$ be the set of $d\times d$ matrices and 
$\Theta $ a non-empty, bounded and closed subset of $\R ^{d\times d}$;
the set $\Theta $ is a collection of parameters 
that represents the variance uncertainty of $G$-Brownian motion.
We associate $\Theta$ with two constants $\sigma_1 ,\sigma_0 \ge 0$ via
\begin{align}\label{;variance}
	\sigma_0 ^2=\inf _{\gamma \in \Theta }
	\inf _{
	\begin{subarray}{c}
	x\in \R ^{d}\\
	|x|=1
	\end{subarray}
	}
	x\cdot \gamma \gamma ^{*}x, 
	&& 
	\sigma_1 ^2=\sup _{\gamma \in \Theta }
	\sup _{
	\begin{subarray}{c}
	x\in \R ^{d}\\
	|x|=1
	\end{subarray}
	}
	x\cdot \gamma \gamma ^{*}x. 
\end{align}
For a $d\times d$ 
symmetric matrix $A$, define 
\begin{flalign}\label{;gen}
	G(A)
	:= 
	\frac{1}{2}
	\sup_{\gamma \in \Theta}
	\tr \left[ A \gamma \gamma ^* \right] . 
\end{flalign}
For a given $\vp \in \bLip (\R ^{d})$, 
we denote by $u_{\vp }$ the unique viscosity 
solution to the following nonlinear partial differential equation 
called the $G$-heat equation:
\begin{align}\label{;gheat}
	\begin{cases}
	\dfrac{\partial u}{\partial t}-G\left( D^2u\right) =0 & 
	\text{in }(0,1)\times \R ^{d}, \\
	u|_{t=0}=\vp & \text{in } \R ^{d}, 
	\end{cases}
\end{align}
where $D^{2}u=\left( \fp{^2u}{ x^{i} \partial x^{j} } \right) _{i,j=1}^{d}$ 
is the Hessian matrix of $u$. 

\begin{rem}\label{hitaika}
	For the existence and uniqueness of the viscosity solution to \eqref{;gheat}, 
	we refer to \cite[Section~C.3]{Peng;10}. 
	If $\sigma_0 >0$, then the solution to 
	\eqref{;gheat} becomes a $C^{1,2}$-solution. 
\end{rem}

It is shown in \cite{Peng;06,Peng;08b} that there exists a unique sublinear 
expectation functional $\E :\bLip (\Omega )\to \R $ 
that possesses the following two properties:
\begin{enumerate}[(i)]{}
	\item for all $0\le s<t\le 1$ and $\vp \in \bLip (\R ^{d})$, 
		\begin{align}
			\E[ \varphi( B_{t}-B_s )] = \E [\varphi(B_{t-s})] = u_{\vp }(t-s,0); 
		\end{align}
	\item for all $n\in \N $, $0\le t_{1}\le \dots \le t_{n}\le 1$ 
		and $\psi \in \bLip ((\R ^{d})^{n})$, 
		\begin{align}\label{e:GE}
			\E \left[ \psi (B_{t_1} , \ldots , B_{t_n})\right] = 
			\E \left[ \psi _1(B_{t_1} , \ldots , B_{t_{n-1}})\right] , 
		\end{align}
		where $\psi _{1}:(\R ^{d})^{n-1}\to \R $ is defined by 
		\begin{align*}
			\psi _{1}(x_1 ,\ldots , x_{n-1})
			=\E \left[ \psi (x_1 ,\ldots , x_{n-1} , B_{t_n}^{t_{n-1}}+ x_{n-1})
			\right] 
		\end{align*}
	with $B^{s}_{t}:=B_{t}-B_{s}$ for $0\le s\le t\le 1$. 
\end{enumerate}
For $t_{k-1} \le t < t_k$, the related conditional expectation 
of $\vp ( B_{t_1} , \dots , B_{t_n} )$ 
on $\bLip (\Omega_t)$ is defined by
\begin{align}\label{e:CGE}
	\E_t [ \vp ( B_{t_1} , \dots , B_{t_n}) ] :=\vp _{n-k} (B_{t_1} , \dots , 
	B_{t_{k-1}} , B_t),
\end{align}
where $\vp _{n-k} (x_1 , \dots , x_{k-1} , x_k) = 
\E [\vp (x_1 , \dots , x_{k-1 }, B^t_{t_k} +x_k , \dots , B^t_{t_n} + x_k )]$. 
The completion of $\bLip (\Omega _{t})$ with respect to the norm 
$\E[| \cdot|]$ is denoted by $\calL^1_G (\Omega _{t})$;
we simply write $\calL^{1}_{G}(\Omega )$ for $\calL^{1}_{G}(\Omega _{1})$.
The functional $\E $ (resp.~$\E _{t}$) is then uniquely extended 
to a sublinear expectation (resp.~a conditional sublinear expectation) 
on $\calL^1_G (\Omega )$.
This extension is called $G$-expectation (resp.~conditional $G$-expectation) 
and will still be denoted by $\E $ (resp.~by $\E _{t}$) in the sequel.
The triplet $(\Omega ,\calL^1_G (\Omega ),\E )$ is called $G$-expectation space,
on which the canonical process $B$ is a $d$-dimensional $G$-Brownian motion;
for more details, we send the reader to \cite{Peng;10} and references therein. 

Let $W= \{ W_{t}= ( W^{1}_{t},\dots ,W^{d}_{t} ) ^{*};t\ge 0 \}$,
together with a probability measure $P $ defined on a suitable measurable space,
be a $d$-dimensional Brownian motion starting from the origin: 
$P (W_{0}=0)=1$.
We denote by $\{ \calF_t \} _{t \ge 0}$ its augmented filtration: 
\begin{align*}
	\calF_t := \sigma ( W_s , 0 \le s \le t) \vee \calN , \quad t \ge 0 , 
\end{align*}
where $\calN$ is the collection of $P $-null events. 
We denote by $\cA$ the set of $\Theta $-valued $\{ \calF_t \} $-progressively 
measurable processes over the interval $[0,1]$. 
For each $\theta \in \cA$,
we denote by $P_{\theta}$ the law of the process 
\begin{align}
	\int_0^{t} \theta _s \,dW_s, \quad 0\le t\le 1,
\end{align}
induced on $\Omega $.
Now we define the capacity $c: \calB (\Omega ) \to [0,1]$ by
\begin{align}
	c \left( A \right) 
	:= 
	\sup_{\theta \in \cA} 
	P_{\theta} (A) 
	\qn{for }A \in \calB (\Omega ) .
	\label{c}
\end{align}
We list some capacity-related terms:
(i)~A set $A\in \calB (\Omega )$ is called polar if $c(A)=0$;
(ii)~a property is said to hold quasi-surely (q.s.) 
if it holds outside a polar set; 
(iii)~a mapping $X : \Omega \to \R $ is said to be quasi-continuous (q.c.) 
if for all $\ve >0$, there exists an open set $O$ with 
$c(O) < \ve$ such that $X|_{O^c}$ is continuous;
(iv)~we say that $X : \Omega \to \R$ has a q.c.\ version 
if there exists a q.c.\ function $Y : \Omega \to \R$ with $X=Y$ q.s.
For each $t\in [0,1]$, we denote by $L^0 (\Omega _{t})$ 
the set of $\calB (\Omega _{t})$-measurable real-valued functions,
and write $L^{0}(\Omega )$ for $L^{0}(\Omega _{1})$.
For each $X \in L^0 (\Omega )$ such that $E_{P_{\theta}} [X]$ exists 
for all $\theta \in \cA$, set
\begin{align}
	\ol{\E} [X] 
	:= 
	\sup_{\theta \in \cA} 
	E_{P_{\theta}} [X]. 
\end{align}
The following characterization of $G$-expectation space 
is given by \cite[Theorem~54]{Denis;10}: 
\begin{align}
	&\calL^1_G (\Omega _{t})
	=
	\left\{ 
	X \in L^0 (\Omega _{t}) : \n{ $X$ has a q.c.\ version, }
	\lim_{n \to \infty} \ol{\E} [|X| \one_{ \{ |X|>n  \} }] =0 
	\right\},
	\label{;char1}\\
	&\E [X] 
	= 
	\ol{\E}[X] 
	\qd \text{for all } X \in \calL^1_G (\Omega ). 
	\label{;char2}
\end{align}
We refer to the latter identity \eqref{;char2} 
as the upper expectation representation for $G$-expectation.
If we denote by $\calP $ the closure of the family 
$\{ P_{\theta} : \theta \in \cA\}$ 
with respect to the topology of weak convergence,
then the same conclusion as \eqref{;char2} holds for the upper expectation 
relative to $\calP$ \cite[Theorem~52]{Denis;10}:
for each $X \in L^0 (\Omega )$ such that $E_{P} [X]$ 
exists for all $P \in \calP $,
set $\hat{\E}[X] := \sup_{P \in \calP } E_{P} [X]$. 
Then
\begin{align}\label{;char3}
	\E [X] 
	= 
	\hat{\E}[X] 
	\qn{for all } X \in \calL^1_G (\Omega ). 
\end{align}
Let us consider another capacity
\begin{align}\label{hatc}
	\hat{c} \left( A \right) 
	:= 
	\sup_{P \in \calP} P ( A ) , 
	\qd A \in \calB ( \Omega ) ,
\end{align}
associated to the upper expectation representation \eqref{;char3}. 
If $N \in \calB ( \Omega )$ is polar under the capacity $c$, 
then $N$ is also polar under the capacity $\hat{c}$.
Indeed, since the indicator function $\one_N$ 
is equal to $0$ q.s.\ under the capacity $c$,
we have in particular that $\one_N (B) \in \calL^1_G ( \Omega )$ 
by \eqref{;char1}.
Then by \eqref{;char3},
we have $\hat{c} \left( N \right) = \hat{\E} \left[ \one_N \right] =0$.
This shows that the quasi-sureness under $\hat{c}$ 
is equivalent to that under $c$. 
Thus we do not need to distinguish these two
and simply write q.s. 

\begin{rem}\label{nimed}
	If variational representations hold 
	under the laws of continuous martingales, 
	then the representation \eqref{eq:intro} is immediate 
	from those and \eqref{;char2}; 
	however, as far as we know, such representations 
	have not been obtained in any existing literature. 
\end{rem}

%%%%%%%%%%%%%%%%%%%%%%%%%%%%%%%%%%%%%%%%%%%%%%%%%%%%%%%%%%%%%%%%%%%%%
%%%%%%%%%%%%%%%%%%%%%%%%%%%%%%%%%%%%%%%%%%%%%%%%%%%%%%%%%%%%%%%%%%%%%
\subsection{Girsanov's formula for $\bm{G}$-Brownian motion}

In order to introduce the statement of Girsanov's formula for 
$G$-Brownian motion from \cite{Osuka;11},
we recall some notions first. 

For each $p\ge 1$ and $t\in [0,1]$,
we denote by $\calL^p_G (\Omega _{t})$ 
the completion of $\bLip (\Omega _{t})$ 
with respect to the norm $\E [| \cdot |^p]^{1/p}$.
When $t=1$, we drop it from notation.
Let
\begin{multline}\label{;mpg}
	M^{p,0}_G (0,1) 
	:=
	\Biggl\{
	\sum_{k=0}^{n-1} \xi_k \one_{[t_k , t_{k+1})} :
	n \in \N ,~ 0= t_0 < t_1 < \dots < t_{n} =1,\\
	\xi_k \in 
	\calL^p_G (\Omega _{t_k}),~k=0,\ldots ,n-1
	\Biggr\}. 
\end{multline}
We denote by $M^p_G (0,1)$ the completion of 
$M^{p,0}_G (0,1)$ under the norm 
\begin{align}
	\| \eta \|_{M^p_G (0,1)}
	:=
	\left\{ 
	\int _{0}^{1} \E 
	\left[ 
	\left| \eta_{t} \right| ^{p} d t
	\right] 
	\right\} ^{1/p}. 
\end{align}
To $\eta =( \eta^i )_{i=1}^d \in ( M^p_G (0,1) )^d$, 
we assign the norm $\| \eta \|_{M^p_G (0,1 ; \R^d )}$ by
\begin{align}
	\| \eta \|_{M^p_G (0,1 ; \R^d )}
	:=
	\| | \eta | \|_{ M^p_G (0,1) }. 
\end{align}

For every $h \in ( \MG )^d$ and $t \in [0,1]$,
the It\^o integral for $G$-Brownian motion
\begin{align}
	\int_0^t h_s \cdot dB_s := \sum_{i=1}^d \int_0^t h^i_s \, dB^i_s 
\end{align}
is defined as an element of $\calL^2_G ( \Omega_t )$.
Here $B^i$ denotes the $i$-th coordinate of $B$.
For $i,j=1 \ddd d$, the mutual variation of $B^i$ and $B^j$
\begin{align}
	\qv{B^i , B^j}_t 
	:= 
	B^i_t B^j_t - \int_0^t B^i_s \, dB^j_s
	- \int_0^t B^j_s \, dB^i_s
\end{align}
is also defined since $B^i$, $B^j$ belong to $\MG$.
We denote the quadratic variation of $B$
by $\qv{B}_t := ( \qv{ B^i , B^j}_t )_{i,j=1}^d$, $0 \le t \le 1$.
In addition, for each $\eta \in ( M^1_G (0,1) )^d$, we define
\begin{align}
	\int_0^t d\qv{B}_s \, \eta_s
	:= 
	\left( 
	\sum_{i=1}^d \int_0^t \eta^i_s \, d\qv{B^1 , B^i}_s
	\ddd 
	\sum_{i=1}^d \int_0^t \eta^i_s \, d\qv{B^d , B^i}_s 
	\right) ^*
\end{align}
as an element of $\left( \calL^1_G (\Omega_t) \right)^d$.
Note that if $\eta^1, \eta^2 \in \MG$, 
then $\eta^1 \eta^2 \in M^1_G (0,1)$.
For each $h \in ( \MG )^d$, we write
\begin{align}
	\int_0^t h_s \cdot \left( d \qv{B}_s \, h_s \right)
	:= 
	\sum_{i,j=1}^d 
	\int_0^t h^i_s h^j_s \, d \qv{B^i , B^j}_s .
\end{align}

\begin{rem}
	From the upper expectation representation \eqref{;char2} 
	and the definition of $P_{\theta}$, $\theta \in \cA$, 
	it is seen that for every $\theta \in \cA$,
	\begin{align}\label{;qv}
		\frac{ d \qv{B}_s }{d s} 
		\in \{ \gamma \gamma^* : \gamma \in \Theta \} 
		\qn{for a.e.\ } s \in [0,1] ~ P_{\theta} \n{-a.s.}
	\end{align}
\end{rem}

An $\R$-valued process $\eta=\{ \eta_t ; 0 \le t \le 1 \}$ 
defined on $( \Omega , \calL^1_G (\Omega) , \E )$ is called a $G$-martingale 
if $\eta_t \in \calL^1_G (\Omega_t)$ for every $0 \le t \le 1$ and
its conditional $G$-expectations satisfy
\begin{align}
	\E_s [ \eta_t ] = \eta_s 
	\qn{in } \calL^1_G ( \Omega_s )
\end{align}
for all $0 \le s \le t$;
$\eta$ is called a symmetric $G$-martingale 
if both $\eta$ and $-\eta$ are $G$-martingales.

With these notions, we now introduce Girsanov's formula 
for $G$-Brownian motion.
Let $h \in ( \MG )^d$.
We define, for $0 \le t \le 1$,
\begin{align}
	&D^{(h)}_t 
	:= 
	\exp 
	\left(
	\int_0^t h_s \cdot d B_s
	- \frac{1}{2} 
	\int_0^t h_s \cdot \left( d \qv{B}_s \, h_s \right) 
	\right), 
	\label{Dh}\\
	&\ol{B}_t 
	:= 
	B_t - \int_0^t d \qv{B}_s \, h_s,
\end{align}
and we set
\begin{align}
	C^{(h)}_{b,Lip} ( \Omega ) 
	:= 
	\{ 
	\vp (\ol{B}_{t_1} \ddd \ol{B}_{t_n} )
	:  n \in \N ,~ t_1 \ddd t_n \in [0,1] ,~ \vp \in \bLip ((\R^d)^n ) 
	\} .
\end{align}

\begin{thm}[\cite{Osuka;11}, Theorem~5.3]\label{;gir}
	Assume that $\sigma_0$ defined by \eqref{;variance} 
	is strictly positive and 
	that $D^{(h)}$ is a symmetric $G$-martingale on 
	$(\Omega , \calL^1_G ( \Omega ) , \E )$.
	Define a sublinear expectation $\E^h$ by
	\begin{align}
		\E^h [X] 
		:= 
		\E [X D^{(h)}_1 ] 
		\qn{for } X \in C^{(h)}_{b,Lip} (\Omega ).
	\end{align}
	Let $\calL^{1, (h)}_G ( \Omega )$ be the completion of 
	$C^{(h)}_{b,Lip} ( \Omega )$ under the norm $\E^h [| \cdot |]$, 
	and extend $\E^h$ to a unique sublinear expectation 
	on $\calL^{1, (h)}_G ( \Omega )$.
	Then the process $\{ \ol{B}_t ; 0 \le t \le 1 \}$ 
	is a $G$-Brownian motion 
	on the sublinear expectation space 
	$( \Omega , \calL^{1, (h)}_G ( \Omega ) , \E^h )$.
\end{thm}

\begin{rem}\label{rem:gir}
	Suppose that the assumptions of \tref{;gir} are fulfilled.
	If a functional $F \equiv F(B)$ on $\Omega$ 
	belongs to $\calL^1_G ( \Omega )$, 
	then we see that by its construction, 
	$\calL^{1, (h)}_G ( \Omega )$ contains $F ( \ol{B} )$
	and then by \tref{;gir},
	\begin{align}
		\E \left[ F(B) \right]
		=
		\E^h \left[ F( \ol{B} ) \right] ;
		\label{eq:gir}
	\end{align}
	this transformation of $G$-expectation will be seen 
	in Sections~\ref{;pr2} and \ref{s:rem}.
\end{rem}

A sufficient condition for $D^{(h)}$ to be a symmetric $G$-martingale,
referred to as $G$-Novikov's condition,
is also given in \cite{Osuka;11}:
there exists $\ve >0$ such that
\begin{align}
	\E 
	\left[ 
	\exp 
	\left( 
	\frac{1}{2} (1+ \ve ) \int_0^1 h_s \cdot (d \qv{B}_s \, h_s) 
	\right) 
	\right] 
	< \infty .
\end{align}

\begin{prop}[\cite{Osuka;11}, Proposition~5.9]\label{;nov}
	If $h \in ( \MG )^d$ satisfies $G$-Novikov's condition, 
	then the process $D^{(h)}$ is a symmetric $G$-martingale.
\end{prop}

In the sequel we denote by $\| \cdot \|_{\infty}$ 
the supremum norm under the capacity $c$: 
\begin{align}
	\| X \|_{\infty}
	:=
	\inf \{ M \ge 0 : c (|X| > M) =0 \}
	\qn{for } X \in L^0 ( \Omega ) .
\end{align}
We say that $h \in ( \MG )^d$ is bounded if
\begin{align}
	\sup_{ 0 \le t \le 1 }
	\| | h_t | \|_{\infty}
	<
	\infty .
\end{align}
$G$-Novikov's condition implies that $D^{(h)}$ 
is a symmetric $G$-martingale if $h$ is bounded.
We also write $\E^h [X]$ for $X \in L^0 ( \Omega )$ 
to denote $\E [ X D^{(h)}_1 ]$ 
whenever $X D^{(h)}_1 \in \calL^1_G ( \Omega )$; 
we recall from \cite[Remark~5.8]{Osuka;11} 
that under the assumptions of \tref{;gir}, 
we have $X D^{(h)}_1 \in \calL^1_G ( \Omega )$ 
if $X \in \calL^{1, (h)}_G ( \Omega )$. 
We close this section with a lemma 
that will also be referred to in Sections~\ref{;pr2} and \ref{s:rem}.

\begin{lem}\label{l:sym}
	Let $h \in ( \MG )^d$ be bounded.
	\begin{enumerate}[{\normalfont (i)}]
	\item
		Let $X \in L^0 ( \Omega )$ be such that 
		$X \in \calL^p_G ( \Omega )$ for some $p>1$.
		Then $\E^h [X]$ is well-defined, that is, 
		$X D^{(h)}_1 \in \calL^1_G ( \Omega )$. 
	\item
		Denote $\ol{B} = B - \int_0^{\cdot} d \qv{B}_s \, h_s$ as above.
		Then it holds that
		\begin{align}
			\left(
			\int_0^1 h_s \cdot d \ol{B}_s
			\right)
			D^{(h)}_1
			\in \calL^1_G ( \Omega )
		\end{align}
		and
		\begin{align}
			\E^h
			\left[
			\int_0^1 h_s \cdot d \ol{B}_s
			\right]
			=
			- \E^h
			\left[
			- \int_0^1 h_s \cdot d \ol{B}_s
			\right]
			= 0 .
		\end{align}
	\end{enumerate}
\end{lem}

\begin{proof}
	(i)
	Let us check $X D^{(h)}_1 \in \calL^1_G ( \Omega )$ 
	via the characterization \eqref{;char1} of $\calL^1_G ( \Omega )$.
	Since $X$ and $D^{(h)}_1$ are in $\calL^1_G ( \Omega )$, 
	it is clear that $X D^{(h)}_1$ has a q.c.\ version.
	Note that by the boundedness of $h$ and \eqref{;qv},
	\begin{align}
		\int_0^1 h_s \cdot \left( d \qv{B}_s \, h_s \right)
		\le
		\sigma_1^2 \sup_{ 0 \le t \le 1 } \| | h_t | \|^2_{\infty}
		\qn{q.s.}
		\label{e:qvint}
	\end{align}
	with the constant $\sigma_1$ given in \eqref{;variance}. 
	Then for any $q \ge 1$,
	\begin{align}
		\ol{\E} \left[ ( D^{(h)}_1 )^q \right]
		&= 
		\ol{\E} 
		\left[ 
		\exp 
		\left( 
		\frac{ q^2 -q }{2} \int_0^1 h_s \cdot \left( d \qv{B}_s \, h_s \right) 
		\right)
		D^{(q h)}_1 
		\right] \\
		&\le 
		\exp 
		\left( 
		\frac{ q^2 -q }{2} \sigma_1^2 \sup_{0 \le t \le 1} \| |h_t| \|_{\infty}^2 
		\right)
		\E \left[ D^{(q h)}_1 \right] \\
		&= 
		\exp 
		\left( 
		\frac{ q^2 -q }{2} \sigma_1^2 \sup_{0 \le t \le 1} \| |h_t| \|_{\infty}^2 
		\right) 
		< 
		\infty,
		\label{eq:u11}
	\end{align}
	where for the last equality, 
	we used the fact that $D^{(qh)}$ is also a symmetric $G$-martingale 
	due to the boundedness of $h$.
	Then by H\"older's inequality,
	\begin{align}
		\ol{\E} \left[ | X D^{(h)}_1 |^{ \frac{p+1}{2} } \right]
		&\le
		\E \left[ |X|^{\frac{p+1}{2} \cdot \frac{2 p}{p+1} } \right] ^{\frac{p+1}{2 p}} 
		\E \left[ 
		( D^{(h)}_1 )^{ \frac{p+1}{2} \cdot \frac{2 p}{p-1} } 
		\right] ^{\frac{p-1}{2 p}} \\
		&=
		\E \left[ |X|^p \right]^{\frac{p+1}{2 p}}
		\E \left[ ( D^{(h)}_1 )^{ \frac{p(p+1)}{p-1} } \right]^{\frac{p-1}{2 p}}
		<
		\infty .
	\end{align}
	Therefore we obtain
	\begin{align}
		\lim_{n \to \infty} \ol{\E} 
		\left[ 
		| X D^{(h)}_1 | \one_{ \{ | X D^{(h)}_1 | \ge n \} } 
		\right] 
		=0 ,
	\end{align}
	and hence $X D^{(h)}_1 \in \calL^1_G ( \Omega )$.
	
	(ii)
	Set
	\begin{align}
		Y_t
		:=
		\int_0^t h_s \cdot d \ol{B}_s
		=
		\int_0^t h_s \cdot d B_s - \int_0^t h_s \cdot 
		\left(
		d \qv{B}_s \, h_s
		\right) 
	\end{align}
	for $0 \le t \le 1$.
	It follows from \eqref{e:qvint} that $Y_1 \in \calL^2_G ( \Omega )$,
	and hence $Y_1 D^{(h)}_1 \in \calL^1_G ( \Omega )$ by (i).
	Since $Y D^{(h)}$ is a $P_{\theta}$-martingale for any $\theta \in \cA$, 
	we have $E_{P_{\theta}} [ Y_1 D^{(h)}_1 ] = - E_{P_{\theta}} [ - Y_1 D^{(h)}_1 ] =0$.
	As $\theta \in \cA$ is arbitrary, we obtain
	\begin{align}
		\E^h \left[ Y_1 \right]
		=
		- \E^h \left[ - Y_1 \right]
		=
		0
	\end{align}
	by the upper expectation representation \eqref{;char2}, 
	and end the proof.
\end{proof}

%%%%%%%%%%%%%%%%%%%%%%%%%%%%%%%%%%%%%%%%%%%%%%%%%%%%%%%%%%%%%%%%%%%%%
%%%%%%%%%%%%%%%%%%%%%%%%%%%%%%%%%%%%%%%%%%%%%%%%%%%%%%%%%%%%%%%%%%%%%
%%%%%%%%%%%%%%%%%%%%%%%%%%%%%%%%%%%%%%%%%%%%%%%%%%%%%%%%%%%%%%%%%%%%%
%%%%%%%%%%%%%%%%%%%%%%%%%%%%%%%%%%%%%%%%%%%%%%%%%%%%%%%%%%%%%%%%%%%%%
%%%%%%%%%%%%%%%%%%%%%%%%%%%%%%%%%%%%%%%%%%%%%%%%%%%%%%%%%%%%%%%%%%%%%
\section{The variational representation for $\bm{G}$-Brownian \\
functionals}\label{;pr2} 

In this section, we state and prove the main result of this paper, 
the variational representation \eqref{eq:intro} 
for functionals on $G$-expectation space. 
In the rest of this paper, we assume $\sigma_0 >0$.
This assumption allows us to apply Girsanov's formula 
for $G$-Brownian motion (\tref{;gir}),
which plays a central role throughout this section and Section~\ref{s:rem}. 

When $f \in \calL^1_G ( \Omega )$ is q.s.\ bounded, we simply call it bounded. 

\begin{thm}\label{;tmain}
	For any bounded elements $f$ of $\calL^1_G ( \Omega )$, 
	it holds that 
	\begin{align}\label{;eqmain}
		\log \E \left[ e^{f(B)} \right]
		=
		\sup _{h \in ( \MG )^d}
		\E 
		\left[ 
		f \left( B+\int _0^{\cdot} d \qv{B}_s \, h_s \right) 
		-
		\frac{1}{2} \int _0^1 h_s \cdot \left( d \qv{B}_s \, h_s \right) 
		\right] . 
	\end{align}
\end{thm}

The well-definedness of the right-hand side 
of \eqref{;eqmain} will be seen in \pref{;comp}.
The following fact will be used in the proof of \tref{;tmain} repeatedly:
\begin{rem}\label{rem:Lip}
	For any $M>0$, the mappings
	\begin{align}\label{lmaps}
		\R \ni x \mapsto 
		\log \left( e^{-M} \vee \left( x \wedge e^M \right) \right)
		& & \n{and} & &
		\R \ni x \mapsto 
		\exp \left\{ \left( -M \right) \vee \left( x \wedge M \right) \right\} 
	\end{align}
	are both Lipschitz functions with Lipschitz constant $e^M$. 
	Here and in what follows we write $x \vee y = \max \{ x,y \}$ 
	and $x \wedge y = \min \{ x,y \}$ for $x,y \in \R$. 
\end{rem}

%%%%%%%%%%%%%%%%%%%%%%%%%%%%%%%%%%%%%%%%%%%%%%%%%%%%%%%%%%%%%%%%%%%%%
%%%%%%%%%%%%%%%%%%%%%%%%%%%%%%%%%%%%%%%%%%%%%%%%%%%%%%%%%%%%%%%%%%%%%
\subsection{A preliminary result}\label{;pr1}

In this subsection, we see that $G$-expectation in the right-hand side 
of \eqref{eq:intro} is well-defined, that is, we prove the following:
\begin{prop}\label{;comp}
	Let $f$ be a bounded element of $\calL^1_G ( \Omega )$. 
	Then, for any $h \in ( \MG )^d$, we have 
	\begin{align}
		f\left( B+\int _{0}^{\cdot }d\qv{B}_{s}\,h_{s}\right) 
		\in \calL^1_G (\Omega ). 
	\end{align}
\end{prop}

A key step to the proof of this proposition 
is an absolute continuity stated in \pref{;abs}.
In what follows, we denote 
\begin{align}
	T^h (B)_t 
	=
	B_t +\int _0^t d \qv{B}_s \, h_s , \qd 
	0 \le t \le 1, 
\end{align}
for $h\in ( \MG )^d$. 

\begin{proof}[Proof of \pref{;comp}]
	Let $\{ f_n \}_{n \in \N} \subset \bLip ( \Omega )$ be such that 
	\begin{align}
		\lim_{n \to \infty} 
		\E [| f_n (B) - f(B) |] =0 . 
	\end{align}
	Truncating $f_n$ if necessary, we may assume that
	$\| f_n \|_{\infty} \le \| f \|_{\infty}$ for all $n \in \N$. 
	Since every $f_n$ is in $\bLip ( \Omega )$ and $\int _0^t d \qv{B}_s \, h_s$ 
	has a q.c.\ version for each $t\in [0,1]$,
	it is clear that the functional $f_n \left( T^h (B) \right) $ 
	also has a q.c.\ version, hence belongs to $\calL^1_G (\Omega )$ 
	due to \eqref{;char1}.
	Therefore, in order to prove the proposition, 
	it is sufficient to show 
	\begin{align}
		\ol{\E}
		\left[ 
		\left| 
		f_n \left( T^h (B) \right)
		- f \left( T^h (B) \right) 
		\right| 
		\right] 
		\xrightarrow[ n \to \infty ]{} 0
		\label{eq1}
	\end{align}
	because of \eqref{;char2}. 
	To this end, fix $\ve > 0$ arbitrarily. 
	By the sublinearity of $\ol{\E}$, 
	the left-hand side of \eqref{eq1} is bounded from above by the sum
	\begin{align}
		&\ol{\E} 
		\left[ 
		\left| f_n ( T^h (B) ) - f( T^h (B) ) \right|
		; T^h (B) \in A_n^c \right] \\
		&\hspace{70pt}
		+
		\ol{\E} 
		\left[ 
		\left| f_n ( T^h (B) ) - f( T^h (B) ) \right| 
		; T^h (B) \in A_n \right] , 
		\label{eq0}
	\end{align}
	where $A_n = \{ \omega : | f_n ( \omega ) - f( \omega ) | > \ve \}$. 
	The first term is less than or equal to $\ve$ from the
	definition of $A_n$. 
	On the other hand, since
	\eqref{abs2} in \pref{;abs} yields the bounds 
	\begin{align}
		\left| f (T^h (B)) \right| ,~ \left| f_n (T^h (B)) \right| 
		\le 
		\| f \|_{\infty}
		\qn{q.s.}
		\label{eq2}
	\end{align}
	for all $n \in \N$, the second term of \eqref{eq0} is
	less than or equal to 
	\begin{align}
		2 \| f \| _{\infty} c \left( T^h(B) \in A_n \right) .
	\end{align}
	By Chebyshev's inequality and \eqref{;char3}, 
	\begin{align}
		\hat{c} \left( A_n \right) 
		\le 
		\ve^{-1}
		\E 
		\left[ 
		| f_n (B) - f(B) |
		\right] ;
		\label{eq3}
	\end{align}
	the right-hand converges to $0$ by letting $n \to \infty$. 
	Therefore \pref{;abs} implies $c \left( T^h(B) \in A_n \right) < \ve$ for
	sufficiently large $n$. 
	Combining these estimates, we can bound 
	\eqref{eq0} from above by
	\begin{align}
		(1+ 2 \| f \| _{\infty} ) \ve 
	\end{align}
	for sufficiently large $n$, 
	and hence obtain \eqref{eq1}. 
\end{proof}

%%%%%%%%%%%%%%%%%%%%%%%%%%%%%%%%%%%%%%%%%%%%%%%%%%%%%%%%%%%%%%%%%%%%%
%%%%%%%%%%%%%%%%%%%%%%%%%%%%%%%%%%%%%%%%%%%%%%%%%%%%%%%%%%%%%%%%%%%%%
\subsection{Proof of the lower bound}\label{;prflb}

In this subsection, we prove the lower bound in \tref{;tmain}:
for any bounded elements $f$ of $\calL^1_G ( \Omega )$,
\begin{align}\label{Lbound}
	\log \E \left[ e^{f(B)} \right]
	\ge 
	\sup_{h \in ( \MG )^d} 
	\E 
	\left[
	f \left( B + \int_0^{\cdot} d \qv{B}_s \, h_s \right)
	- 
	\frac{1}{2} \int_0^1 h_s \cdot \left( d \qv{B}_s \, h_s \right) 
	\right] .
\end{align}

\begin{lem}\label{lem:l1}
	Let $h \in ( \MG )^d$ be bounded. 
	Then for any bounded $f \in \calL^1_G ( \Omega )$, we have
	\begin{align}\label{eq:krhr}
		\log \E \left[ e^{f(B)} \right]
		\ge 
		\E ^h \left[ f(B) - \log D^{(h)}_1 \right].
	\end{align}
\end{lem}

\begin{proof}
	First note that $f(B) - \log D^{(h)}_1 \in \calL^2_G ( \Omega )$ 
	by the assumption. 
	Hence the right-hand side of \eqref{eq:krhr} 
	is well-defined by (i) of \lref{l:sym}, 
	and is equal to
	\begin{align}
		&\E^h 
		\left[ 
		\log \E \left[ e^{f(B)} \right] 
		-
		\log \E \left[ e^{f(B)} \right]
		+ 
		\log e^{f(B)} 
		- \log D^{(h)}_1 
		\right] \\
		&= 
		\log \E \left[ e^{f(B)} \right]
		+ 
		\E^h 
		\left[ 
		- \log 
		\left( 
		\frac{ \E \left[ e^{f(B)} \right] }{ e^{f(B)} } 
		D^{(h)}_1 
		\right)
		\right] . 
		\label{eq:l1}
	\end{align}
	We rewrite the second term of \eqref{eq:l1} to obtain the bound
	\begin{align}
		&\E 
		\left[ 
		- \frac{ \E \left[ e^{f(B)} \right] }{ e^{f(B)} } 
		D^{(h)}_1
		\log 
		\left( 
		\frac{ \E \left[ e^{f(B)} \right] }{ e^{f(B)} } 
		D^{(h)}_1 
		\right)
		\times 
		\frac{ e^{f(B)} }{ \E \left[ e^{f(B)} \right] }
		\right] \\
		&\le 
		\E 
		\left[ 
		\left( 
		1- 
		\frac{ \E \left[ e^{f(B)} \right] }{ e^{f(B)} } 
		D^{(h)}_1 
		\right)
		\times 
		\frac{ e^{f(B)} }{ \E \left[ e^{f(B)} \right] } 
		\right] \\
		&=
		\E 
		\left[ 
		\frac{ e^{f(B)} }{ \E \left[ e^{f(B)} \right] } 
		- D^{(h)}_1 
		\right] ,
	\end{align}
	where the inequality follows from that 
	$-x \log x \le 1-x$ for $x>0$.
	Since $D^{(h)}$ is a symmetric $G$-martingale 
	by the boundedness of $h$
	and $G$-Novikov's condition (\pref{;nov}), 
	\begin{align}
		\E 
		\left[ 
		\frac{ e^{f(B)} }{ \E \left[ e^{f(B)} \right] } 
		- D^{(h)}_1 
		\right]
		= 
		\E 
		\left[ 
		\frac{ e^{f(B)} }{ \E \left[ e^{f(B)} \right] } 
		\right] 
		+ 
		\E \left[ -D^{(h)}_1 \right]
		= 1-1
		= 0.
	\end{align}
	Therefore the lemma follows.
\end{proof}

We denote by $\sbl$ the subset of $M^{2,0}_G (0,1)$ 
consisting of elements with each $\xi _k$ 
in the definition \eqref{;mpg} of $M^{p,0}_G (0,1)$ 
belonging to $\bLip ( \Omega _{t_k} )$.
Since $\bLip ( \Omega _{t_k} )$ is dense in $\calL ^2_G ( \Omega _{t_k} )$, 
we may deduce that $\sbl$ is also dense in $\MG$.
Let $h \equiv h_{\cdot} (B) \in ( \sbl )^d$ be written as 
\begin{align}
	h_t
	=
	\sum _{k=0}^{n-1} 
	\xi _k \one _{[ t_k , t_{k+1} )} (t), 
	\qd 0 \le t \le 1, 
	\label{sbl}
\end{align}
for some $n \in \N $, $0=t_0 < t_1 < \dots < t_n =1$, 
$\xi _0 \in \R ^d$, 
and $\xi _k \equiv \xi _k (B) \in ( \bLip ( \Omega _{t_k} ))^d$, 
$k=1 \ddd n-1$.
We associate $h$ with a simple process $\OL{h}$ defined as follows: 
\begin{align}
	&\bar{\xi} _0 
	:= \xi _0, 
	& &
	\OL{h}_t 
	:= \bar{\xi} _0 ,~ 
	t_0 \le t < t_1 , \\
	&\bar{\xi} _1 
	:= \xi _1 
	\left( 
	B_t - \int _0^t d \qv{B}_s \, \OL{h}_s , \, t \le t_1 
	\right) , 
	& &
	\OL{h} _t := \bar{\xi} _1 ,~ 
	t_1 \le t < t_2 , \\
	&\hspace{14em}\vdots \\
	&\bar{\xi}_{n-1} 
	:= \xi _{n-1} 
	\left( 
	B_t - \int _0^t d \qv{B} _s \, \OL{h} _s , \, t \le t_{n-1} 
	\right) , 
	& &
	\OL{h} _t 
	:= \bar{\xi}_{n-1} ,~
	t_{n-1} \le t < t_n . 
\end{align}
By this construction, it is clear that each $\bar{\xi} _k$ 
belongs to $\calL^1_G ( \Omega )$ 
and is bounded, and hence $\OL{h}$ 
is a bounded element of $( M^{2,0}_G (0,1) )^d$;
moreover, 
\begin{align}\label{;barh}
	\OL{h}_t = h_t \left( T^{-\OL{h}} (B)\right) 
	\quad \text{for all } 0 \le t \le 1. 
\end{align}

\begin{lem}\label{lem:lbLip}
	\eqref{Lbound} holds for $f \in \bLip ( \Omega )$.
\end{lem}

\begin{proof}
	It is sufficient to show
	\begin{align}\label{eq:bLiph}
		\log \E \left[ e^{f(B)} \right]
		\ge  
		\E 
		\left[
		f \left( T^h (B) \right)
		- \frac{1}{2} \int_0^1 
		h_s \cdot \left( d \qv{B}_s \, h_s \right) 
		\right]
	\end{align}
	for all $h \in ( \MG )^d$.

	First we let $h \in ( \sbl )^d$. 
	Let $\OL{h}$ be the associated element of $( M^{2,0}_G (0,1) )^d$ 
	as constructed above so that \eqref{;barh} holds. 
	By the boundedness of $\OL{h}$, \lref{lem:l1} implies
	\begin{align}
		\log \E \left[ e^{f(B)} \right]
		&\ge 
		\E^{\OL{h}} 
		\left[ 
		f(B)
		- \frac{1}{2} \int_0^1 
		\OL{h}_s \cdot \left( d \qv{B}_s \, \OL{h}_s \right)
		- \int_0^1 \OL{h}_s \cdot d \ol{B}_s \right] \\
		&= 
		\E^{\OL{h}} 
		\left[ 
		f(B)
		- \frac{1}{2} \int_0^1 
		\OL{h}_s \cdot \left( d \qv{B}_s \, \OL{h}_s \right) 
		\right] ,
	\end{align}
	where $\ol{B} := T^{-\OL{h}} (B)$ 
	and the equality follows from (ii) of \lref{l:sym}. 
	By \eqref{;barh} and the obvious identity $\qv{B} = \qv{\ol{B}}$, 
	we can rewrite the right-hand side as 
	\begin{align}
		\E^{\OL{h}}
		\left[ 
		f \left( \ol{B} + \int_0^{\cdot} d \qv{\ol{B}}_s \, h_s ( \ol{B} ) \right)
		- 
		\frac{1}{2} \int_0^1 
		h_s ( \ol{B} ) \cdot \left( d \qv{ \ol{B} }_s \, h_s ( \ol{B} ) \right) 
		\right] .
		\label{olh}
	\end{align}
	By the boundedness of $\OL{h}$, 
	\pref{;nov} implies that $D^{( \OL{h} )}$ 
	is a symmetric $G$-martingale, 
	and hence by Girsanov's formula for $G$-Brownian motion (\tref{;gir}), 
	$\ol{B}$ is a $G$-Brownian motion on 
	$( \Omega , \calL^{1, (\OL{h})}_G ( \Omega ) , \E^{\OL{h}})$.
	Since $f \left( B + \int_0^{\cdot} d \qv{B}_s \, h_s (B) \right)$
	is in $\calL^1_G ( \Omega )$ by the assumption $f \in \bLip ( \Omega )$ 
	and $\int_0^1 h_s (B) \cdot \left( d \qv{B}_s \, h_s (B) \right)$
	is also in $\calL^1_G ( \Omega )$, 
	we see from \eqref{eq:gir} in \rref{rem:gir} that \eqref{olh} is equal to
	\begin{align} 
		\E 
		\left[ 
		f \left( B+ \int_0^{\cdot} d\qv{B}_s \, h_s \right)
		-
		\frac{1}{2} \int_0^1 h_s \cdot \left( d\qv{B}_s \, h_s \right) 
		\right] .
	\end{align}
	Therefore \eqref{eq:bLiph} holds for $h \in ( \sbl )^d$.

	Now we let $h \in ( \MG )^d$ and take a sequence 
	$\{ h^n \}_{n \in \N} \subset ( \sbl )^d$ such that
	\begin{align}\label{eq:l4}
		\left\| h - h^n \right\|_{ \MGR } \to 0 \qn{as } n \to \infty .
	\end{align}
	As seen above, \eqref{eq:bLiph} holds for every $h^n$, 
	from which it follows that
	\begin{align}
			\log \E \left[ e^{f(B)} \right]
			\ge \,
			&\E \left[
			f \left( T^h (B) \right)
			-\frac{1}{2} \int_0^1 h_s \cdot \left( d\qv{B}_s \, h_s \right)
			\right] \\
			&- \E \left[
			f \left( T^h (B) \right)
			-f \left( T^{h^n} (B) \right)
			\right] \\
			&- \frac{1}{2} \E \left[
			\int_0^1 h^n_s \cdot \left( d\qv{B}_s \, h^n_s \right)
			- \int_0^1 h_s \cdot \left( d\qv{B}_s \, h_s \right)
			\right] .
		\label{eq:l3}
	\end{align}
	Therefore it suffices to show that the second and third terms 
	in the right-hand side of \eqref{eq:l3} tend to $0$ as $n \to \infty$.
	For the second term, since $f$ is Lipschitz,
	\begin{align}
		\left| \E \left[
		f \left( T^h (B) \right)
		-f \left( T^{h^n} (B) \right)
		\right] \right|
		&\le \lip (f) \ol{\E} \left[ \sup_{ 0 \le t \le 1} \left|
		\int_0^t d\qv{B}_s \left( h_s - h^n_s \right) 
		\right| \right] \\
		&\le \lip (f) \sigma_1^2 
		\left\|  h - h^n \right\|_{ \MGR } , \label{eq:l5}
	\end{align}
	where in the last line, 
	we use \eqref{;qv} and the Cauchy-Schwarz inequality. 
	On the other hand, for the third term, we also have
	\begin{align}
		&\left| \E \left[
		\int_0^1 h^n_s \cdot \left( d\qv{B}_s \, h^n_s \right)
		- \int_0^1 h_s \cdot \left( d\qv{B}_s \, h_s \right)
		\right] \right| \\
		&\le
		\E \left[ \left|
		\int_0^1
		\left( h^n_s + h_s \right)
		\cdot \left( d\qv{B}_s \left( h^n_s - h_s \right) \right)
		\right| \right] \\
		&\le
		\sigma^2_1
		\| h^n + h \|_{ \MGR }
		\| h^n - h \|_{ \MGR } .
		\label{eq:l6}
	\end{align}
	Since $\{ h^n \}$ is an approximate sequence of $h$, 
	\eqref{eq:l5} and \eqref{eq:l6} tend to $0$ as $n \to \infty$.
	Therefore \eqref{eq:bLiph} is valid for $h \in ( \MG )^d$ 
	and we complete the proof.
\end{proof}

We are ready to prove the lower bound in \tref{;tmain}.

\begin{prop}\label{lem:lCb}
	\eqref{Lbound} holds for all bounded elements $f$ of $\calL^1_G ( \Omega )$.
\end{prop}

\begin{proof}
	Fix $h \in ( \MG )^d$ arbitrarily
	and let $\{ f_n \}_{n \in \N}$ be as in the proof of \pref{;comp}. 
	\lref{lem:lbLip} implies that
	\begin{align}
		&\log \E \left[ e^{f(B)} \right]
		- \E 
		\left[ 
		f \left( T^h (B) \right)
		- \frac{1}{2} \int_0^1 h_s \cdot \left( d \qv{B}_s \, h_s \right) 
		\right] \\
		&\ge 
		\log \E \left[ e^{f(B)} \right] 
		- \log \E \left[ e^{ f_n (B)} \right]
		- \E 
		\left[
		\left| 
		f \left( T^h (B) \right) - f_n \left( T^h (B) \right) 
		\right| 
		\right] .
		\label{eq:krrn}
	\end{align}
	As seen in the proof of \pref{;comp}, 
	the third term in the right-hand side converges to $0$ as $n \to \infty$. 
	By taking $M = \| f \|_{\infty}$ in \rref{rem:Lip}, we also have
	\begin{align}
		\left|
		\log \E \left[ e^{f(B)} \right] 
		- 
		\log \E \left[ e^{ f_n (B)} \right]
		\right| 
		\le 
		e^{ 2 \| f \|_{\infty} } 
		\E \left[ | f(B) - f_n (B) | \right]
		\to 0 
		\label{expLip}
	\end{align} 
	as $n \to \infty$, and hence obtain the proposition. 
\end{proof}

%%%%%%%%%%%%%%%%%%%%%%%%%%%%%%%%%%%%%%%%%%%%%%%%%%%%%%%%%%%%%%%%%%%%%
%%%%%%%%%%%%%%%%%%%%%%%%%%%%%%%%%%%%%%%%%%%%%%%%%%%%%%%%%%%%%%%%%%%%%
\subsection{Proof of the upper bound}\label{;prfub}

In this subsection, we prove the upper bound in \tref{;tmain}:
for any bounded elements $f$ of $\calL^1_G ( \Omega )$,
\begin{align}\label{Ubound}
	\hspace*{-5mm}
	\log \E \left[ e^{f(B)} \right]
	\le 
	\sup_{h \in ( \MG )^d} \E \left[
	f \left( B + \int_0^{\cdot} d \qv{B}_s \, h_s \right)
	- \frac{1}{2} \int_0^1 
	h_s \cdot \left( d \qv{B}_s \, h_s \right) \right].
\end{align}

Proofs of the next two lemmas proceed 
as those of Lemma~2.1 and Theorem~2.2 in Chapter~IV of \cite{Peng;10}.
\begin{lem}\label{lem:u1}
	Let $0 \le t_1 \le 1$.
	For every $\vp \in \bLip ( \R^d )$, there exist a bounded $h \in ( \MG )^d$ 
	and an $A \in \calL^2_G ( \Omega )$ with $A \ge 0$ q.s.\ such that
	\begin{align}
		\vp (B_1)
		= 
		\log \E_{t_1} \left[ e^{ \vp (B_1)} \right]
		+ \int_{t_1}^1 h_s \cdot d B_s
		- \frac{1}{2} \int_{t_1}^1 h_s \cdot \left( d \qv{B}_s \, h_s \right)
		- A 
	\end{align}
	in $\calL^2_G ( \Omega)$, 
	where $\E_{t_1}$ is the conditional $G$-expectation defined by \eqref{e:CGE}.
\end{lem}

\begin{proof}
	If we let
	\begin{align}
		u(t,x)
		:= \E \left[ e^{ \vp (B_1 - B_t +x) } \right]
	\end{align}
	for $(t,x) \in [0,1] \times \R^d$, 
	then by the assumption $\sigma_0 >0$, 
	$u$ is the $C^{1,2} \left( (0,1) \times \R^d \right)$-solution of
	\begin{align}\label{eq:bG}
		\left\{
		\begin{aligned}
		&\fp{u}{t} + G \left( D^2 u \right) =0 
		& & \n{in } (0,1) \times \R^d , \\
		&u (1,x) = e^{\vp (x)} 
		& & \n{for } x \in \R^d ;
		\end{aligned}
		\right.
	\end{align}
	see \rref{hitaika}.
	Observe that by letting $| \vp | := \sup_{x \in \R^d} | \vp (x) |$, 
	\begin{align}\label{eq:u0}
		0< e^{- | \vp |} \le u(t,x) \le e^{| \vp |} < \infty
	\end{align}
	for any $(t,x)$, and that by the Lipschitz continuity of $\vp$,
	\begin{align}\label{eq:u1}
		\sup_{(t,x) \in (0,1) \times \R^d} \left| \nabla u(t,x) \right| < \infty,
	\end{align}
	where $\nabla := \left( \fp{}{x^1} \ddd \fp{}{x^d} \right)^*$.
	Additionally,
	we note that by Theorem~4.5 in Appendix~C in \cite{Peng;10},
	there exists an $\alpha \in (0,1)$ such that
	\begin{align}\label{eq:u2}
		\| u \|
		_{C^{1+ \frac{\alpha}{2} ,2+ \alpha } ( [0,1- \ve ] \times \R^d )}
		< \infty \qn{for every } \ve \in (0,1).
	\end{align}
	Here
	\begin{align}
		&\| u \|
		_{C^{ \frac{\alpha}{2} , \alpha } ( [0,1- \ve ] \times \R^d )}
		:= 
		\sup_{
		\begin{subarray}{c} 
			x,y \in \R^d , \ x \neq y \\ s,t \in [0,1- \ve ], \ s \neq t 
		\end{subarray}
		}
		\frac{ \left| u(s,x) - u(t,y) \right| }
		{ \left| s-t \right|^{\alpha /2} + \left| x-y \right|^{\alpha} }, \\
		&\| u \|
		_{C^{1+ \frac{\alpha}{2} , 2 + \alpha } ( [0,1- \ve ] \times \R^d )}\\
		&:= 
		\| u \|
		_{C^{ \frac{\alpha}{2} , \alpha } ( [0,1- \ve ] \times \R^d )}
		+ 
		\left\| \fp{u}{t} \right\|
		_{C^{ \frac{\alpha}{2} , \alpha } ( [0,1- \ve ] \times \R^d )}
		+ 
		\sum_{i=1}^d \left\| \fp{u}{x^i} \right\|
		_{C^{ \frac{\alpha}{2} , \alpha } ( [0,1- \ve ] \times \R^d )} \\
		&+ 
		\sum_{i,j=1}^d \left\| \fp{^2 u}{x^i \partial x^j} \right\|
		_{C^{ \frac{\alpha}{2} , \alpha } ( [0,1- \ve ] \times \R^d )}.
	\end{align}
	Set $U (t,x) := \log u (t,x)$.
	Then $U$ is also a member of $C^{1,2} \left( (0,1) \times \R^d \right)$ and
	\begin{align}\label{eq:u3}
		\left\{
		\begin{aligned}
			&\fp{U}{t} (t,x) = \frac{1}{u (t,x)} \fp{u}{t} (t,x) ,\\
			&\nabla U(t,x) = \frac{1}{u (t,x)} \nabla u(t,x), \\
			&D^2 U(t,x) = -\nabla U(t,x) \left( \nabla U(t,x) \right)^*
			+ \frac{1}{u(t,x)} D^2 u(t,x) .
		\end{aligned}
		\right.
	\end{align}
	Therefore we have by \eqref{eq:u0} and \eqref{eq:u2},
	\begin{align}
		\| U \|
		_{C^{1+ \frac{\alpha}{2} ,2+ \alpha } ( [0,1- \ve ] \times \R^d )} < \infty
		\qn{for any } \ve \in (0,1),
	\end{align}
	which allows us to apply $G$-It\^o's formula 
	\cite[Theorem~6.5 in Chapter~III]{Peng;10} 
	to $U(t, B_t)$ on $[t_1, 1- \ve ]$ for $0< \ve < 1-t_1$.
	Then we have
	\begin{align}
		U(1- \ve , B_{1- \ve})
		&= U(t_1 , B_{t_1})
		+ \int_{t_1}^{1- \ve} \fp{U}{t} (s, B_s) \, d s
		+ \int_{t_1}^{1- \ve} \nabla U(s, B_s) \cdot d B_s\\
		&+ \frac{1}{2} \int_{t_1}^{1- \ve} 
		\tr \left[ d \qv{B}_s \, D^2 U(s, B_s) \right]
		\qn{in } \calL^2_G ( \Omega ).
	\end{align}
	By \eqref{eq:u3}, together with \eqref{eq:bG}, we obtain
	\begin{align}\label{eq:u4}
		\begin{aligned}
			&U(1- \ve , B_{1- \ve})\\
			&= U(t_1 , B_{t_1})
			+ \int_{t_1}^{1- \ve} h_s \cdot d B_s
			- \frac{1}{2} \int_{t_1}^{1- \ve} 
			h_s \cdot \left( d \qv{B}_s \, h_s \right)
			- A_{1- \ve} ,
		\end{aligned}
	\end{align}
	where
	\begin{align}
		&h_t := \nabla U(t, B_t), \qd 0 \le t < 1,\\
		&A_t := \int_{t_1}^t
		\frac{ G \left( D^2 u(s, B_s) \right) }{ u(s, B_s)} d s
		- \int_{t_1}^t
		\frac{1}{2}
		\frac{ \tr \left[ d \qv{B}_s \, D^2 u(s, B_s) \right] }{ u(s, B_s)} ,
		\qd t_1 \le t \le 1.
	\end{align}
	Note that $\{ A_t ; t_1 \le t \le 1 \}$ 
	is a nondecreasing process with $A_{t_1} =0$
	and each term of \eqref{eq:u4} is an element of $\calL^2_G ( \Omega )$.
	For the left-hand side of \eqref{eq:u4}, we first note that by \rref{rem:Lip},
	\begin{align}
		&\left| u(s,x) - u(t,x) \right|
		\le
		e^{| \vp |} \lip ( \vp ) 
		\E \left[ \left| B_s - B_t \right| \right] , \\
		&\left| u(t,x) - u(t,y) \right|
		\le
		e^{| \vp |} \lip ( \vp ) |x-y|
	\end{align}
	for all $0 \le s,t \le 1$ and $x,y \in \R^d$. 
	Again by \rref{rem:Lip} and by the above estimates, 
	\begin{align}
		&\left| U(1, B_1) - U(1- \ve , B_{1- \ve}) \right| \\
		&\le e^{| \vp |} \left| u(1, B_1) - u(1- \ve , B_{1- \ve}) \right| \\
		&\le e^{2 | \vp |} \lip ( \vp ) \Big\{
		\E \left[ \left| B_1 - B_{1- \ve} \right| \right] 
		+ \left| B_1 - B_{1- \ve} \right| \Big\}\\
		& \to 0 \qn{as } \ve \to 0 \qn{in } \calL^2_G ( \Omega ). \label{eq:u5}
	\end{align}
	For the second and third terms in the right-hand side of \eqref{eq:u4},
	since $\nabla U$ is bounded because 
	of \eqref{eq:u0}, \eqref{eq:u1} and \eqref{eq:u3}, we have
	\begin{align}\label{eq:u6}
		\int_{1- \ve}^1 h_s \cdot d B_s \to 0 \qn{and} \qd
		\int_{1- \ve}^1 h_s \cdot \left( d \qv{B}_s \, h_s \right) \to 0
	\end{align}
	as $\ve \to 0$ in $\calL^2_G ( \Omega )$.
	For the convergence of $A_{1- \ve}$,
	it is clear that $A_{1- \ve} \to A := A_1$ as $\ve \to 0$ 
	$P_{\theta}$-a.s.\ for every $\theta \in \cA$.
	Moreover, by \eqref{eq:u5} and \eqref{eq:u6},
	\begin{align}
		A_{1- \ve} \to
		- U(1, B_1) + U(t_1 , B_{t_1})
		+ \int_{t_1}^1 h_s \cdot dB_s
		- \frac{1}{2} \int_{t_1}^1 h_s \cdot \left( d \qv{B}_s \, h_s \right)
	\end{align}
	as $\ve \to 0$ in $\calL^2_G ( \Omega )$.
	For every $\theta \in \cA$, 
	taking a subsequence if necessary, we see that this convergence also holds 
	$P_\theta$-a.s.
	Therefore we have
	\begin{align}\label{eq:u7}
		\hspace*{-3mm}
		A = -U(1, B_1) + U(t_1 , B_{t_1})
		+ \int_{t_1}^1 h_s \cdot dB_s
		- \frac{1}{2} \int_{t_1}^1 h_s \cdot \left( d \qv{B}_s h_s \right)
	\end{align}
	$P_{\theta}$-a.s.\ for all $\theta \in \cA$.
	Since all terms in the right-hand side of \eqref{eq:u7} are in $\calL^2_G ( \Omega )$, 
	$A$ is also in $\calL^2_G ( \Omega )$.
	As a consequence, the equality \eqref{eq:u7} holds in $\calL^2_G ( \Omega )$.
	By noting that
	\begin{align}
		&U(1, B_1) = \vp (B_1), \\
		&U(t_1, B_{t_1})
		= \log \E \left. \left[ e^{\vp (B_1 - B_{t_1} +x)} \right] \right|_{x = B_{t_1}}
		= \log \E_{t_1} \left[ e^{\vp (B_1)} \right],
	\end{align}
	we finally obtain
	\begin{align}
		\vp (B_1) = \log \E_{t_1} \left[ e^{\vp (B_1)} \right]
		+ \int_{t_1}^{1} h_s \cdot d B_s
		- \frac{1}{2} \int_{t_1}^1 h_s \cdot \left( d \qv{B}_s \, h_s \right)
		- A
	\end{align}
	in $\calL^2_G ( \Omega )$.
\end{proof}

The following lemma is a type of the Clark-Ocone formula 
in the framework of $G$-expectation space.

\begin{lem}\label{lem:u2}
	For every $f \in \bLip ( \Omega )$, there exist a bounded $h \in ( \MG )^d$
	and an $A \in \calL^2_G ( \Omega )$ with $A \ge 0$ q.s.\ such that
	\begin{align}\label{eq:it}
		f(B)
		= \log \E \left[ e^{ f(B) } \right]
		+ \int_0^1 h_s \cdot d B_s
		- \frac{1}{2} \int_0^1 h_s \cdot \left( d \qv{B}_s \, h_s \right)
		- A
		\qn{in } \calL^2_G ( \Omega ).
	\end{align}
\end{lem}

\begin{proof}
	Let $0 \le t_1 \le 1$ and $\vp \in \bLip ((\R^d)^2)$.
	It suffices to show the lemma holds when $f(B) = \vp (B_{t_1} , B_1)$.
	Set
	\begin{align}
		u(t,x,y) := \E \left[ e^{\vp (x, B_1- B_t +y)} \right], 
		\qd U(t,x,y) := \log u(t,x,y)
	\end{align}
	for $(t,x,y) \in [0,1] \times \R^d \times \R^d$.
	By \lref{lem:u1}, we have for every $x \in \R^d$,
	\begin{align}
		\hspace*{-2mm}
		U(1,x, B_1)
		= 
		&\log \E 
		\left. \left[ 
		e^{\vp (x, B_1 - B_{t_1} +y)} 
		\right] \right|_{y= B_{t_1}} 
		+ \int_{t_1}^1 \nabla_y U(s,x, B_s) \cdot d B_s \\
		&-\frac{1}{2} \int_{t_1}^1 
		\nabla_y U(s,x, B_s) \cdot \left( d \qv{B}_s \, \nabla_y U(s,x, B_s) \right)\\
		&- 
		\left( 
		\int_{t_1}^1
		\frac{ G \left( D^2_y u(s,x, B_s) \right) }{ u(s,x, B_s) } \, d s
		- \int_{t_1}^1 \frac{1}{2} 
		\frac{ \tr \left[ d \qv{B}_s \, D^2_y u(s,x, B_s) \right] }{ u(s,x, B_s) } 
		\right)
	\end{align}
	in $\calL^2_G ( \Omega )$,
	where $\nabla_y := \left( \fp{}{y^1} \ddd \fp{}{y^d} \right)^*$ 
	and $D^2_y := \left( \fp{^2}{y^i \partial y^j} \right)_{i,j =1}^d$.
	By the construction of the integrations 
	with respect to $dB_s$ and $d \qv{B}_s$ (see, e.g., \cite{Peng;08b}),
	this identity still holds with $x$ replaced by $B_{t_1}$.
	Hence, by letting
	\begin{align}
		&\vp_1 (x) := \log \E \left[ e^{\vp (x, B_1 - B_{t_1} +x)} \right], \\
		&h_s := \nabla_y U(s, B_{t_1} , B_s), ~ t_1 \le s <1 ,\\
		&A^{(1)} := \int_{t_1}^1
		\frac{ G \left( D^2_y u(s, B_{t_1} , B_s) \right) }{ u(s, B_{t_1} , B_s) } \, d s
		- \int_{t_1}^1 \frac{1}{2} 
		\frac{ \tr \left[ d \qv{B}_s \, D^2_y u(s,B_{t_1},B_s) \right] }{ u(s,B_{t_1},B_s) } ,
	\end{align}
	we get
	\begin{align}
		\vp (B_{t_1} , B_1)
		&= U(1, B_{t_1} , B_1)\\
		&= \vp_1 (B_{t_1})
		+ \int_{t_1}^1 h_s \cdot d B_s
		- \frac{1}{2} \int_{t_1}^1 h_s \cdot \left( d \qv{B}_s \, h_s \right)
		- A^{(1)} .
	\end{align}
	Since $\vp_1 \in \bLip ( \R^d )$, 
	we may apply \lref{lem:u1} to $\vp_1 (B_{t_1})$
	to obtain $h \in ( \MG )^d$ and $A^{(2)} \in \calL^2_G ( \Omega )$ 
	with $A^{(2)} \ge 0$ q.s.\ such that
	\begin{align}
		&\vp (B_{t_1} , B_1) \\
		&= \log \E \left[ e^{\vp_1 (B_{t_1})} \right]
		+ \int_0^1 h_s \cdot d B_s
		- \frac{1}{2} \int_0^1 h_s \cdot \left( d \qv{B}_s \, h_s \right)
		- \left( A^{(1)} + A^{(2)} \right)
	\end{align}
	in $\calL^2_G ( \Omega )$.
	Noting \eqref{e:GE}, we have 
	\begin{align}
		\E \left[ e^{\vp_1 (B_{t_1})} \right]
		= 
		\E \left[ 
		\E \left. \left[ e^{\vp (x, B_1 - B_{t_1} +x)} \right] \right|_{x=B_{t_1}}
		\right]
		= 
		\E \left[ e^{\vp (B_{t_1} , B_1)} \right] .
	\end{align}
	Therefore the lemma follows.
\end{proof}

Recall that $\calP$ is the weak closure of 
$\{ P_{\theta} : \theta \in \cA \}$; 
by the tightness of $\{ P_{\theta} : \theta \in \cA \}$ 
(\cite[Proposition~50]{Denis;10}), $\calP$ is weakly compact. 

\begin{lem}\label{urem}
	For $f \in \bLip ( \Omega )$, 
	let $A \in \calL^2_G ( \Omega )$ be given in \lref{lem:u2}. 
	Then there exists $P \in \calP$ such that 
	\begin{align}
		A = 0
		\qd P \n{-a.s.}
		\label{Pn}
	\end{align}
\end{lem}

\begin{proof}
	Rewriting \eqref{eq:it} in \lref{lem:u2}, we have
	\begin{align}
		e^{ f(B) + A }
		=
		D^{(h)}_1 \E \left[ e^{f(B)} \right] . 
	\end{align}
	By the boundedness of $h$ and \pref{;nov}, 
	$D^{(h)}$ is a symmetric $G$-martingale and hence 
	the right-hand side is symmetric. 
	Therefore the left-hand side is also symmetric, which yields
	\begin{align}
		\E \left[ e^{ f(B) + A } \right]
		=
		- \E \left[ -e^{ f(B) + A } \right] 
		=
		\E \left[ e^{f(B)} \right] . 
		\label{si}
	\end{align}
	Since every integrand in \eqref{si} is an element of $\calL^1_G ( \Omega )$, 
	we see from \eqref{;char3} that 
	\eqref{si} still holds if $\E$ is replaced by $\hat{\E}$, 
	from which it follows that 
	\begin{align}
		E_P \left[ e^{ f(B) +A } \right]
		=
		\hat{\E} \left[ e^{f(B)} \right] 
		\qn{for all } P \in \calP .
		\label{A2}
	\end{align}
	Also observe that, since $f$ is bounded and continuous, 
	the mapping $\calP \ni P \mapsto E_P [ e^{f(B)} ]$ 
	is continuous by the definition of weak convergence. 
	Then by the compactness of $\calP$, 
	there exists $P^{\prime} \in \calP$ which attains the supremum 
	of $E_P [e^{f(B)}]$ over $P \in \calP$, namely 
	\begin{align}
		E_{P^{\prime}} \left[ e^{f(B)} \right]
		=
		\hat{\E} \left[ e^{f(B)} \right] .
		\label{A1}
	\end{align}
	Combining \eqref{A2} and \eqref{A1} leads to $P^{\prime} (A=0) =1$
	since $A$ is nonnegative q.s.
\end{proof}

In the case of a classical Brownian motion, the following lemma 
is a consequence of Scheff\'e's lemma, 
the equivalence between $L^1$-convergence 
and the convergence of $L^1$-norms 
for an a.s.\ convergent sequence of random variables.
Although the setting is restricted, 
this lemma may be regarded 
as a sublinear counterpart to Scheff\'e's lemma.

\begin{lem}\label{lem:sche}
	For a bounded $h \in ( \MG )^d$, 
	let $\{ h^n \}_{n \in \N} \subset ( \sbl )^d$ be such that
	$\sup_{0 \le t \le 1} \| | h^n_t | \|_{\infty} \le \sup_{0 \le t \le 1} \| | h_t | \|_{\infty}$
	and
	\begin{align}\label{eq:hn}
		\lim_{n \to \infty} \left\| h - h^n \right\|_{ \MGR } =0.
	\end{align}
	Then we have
	\begin{align}\label{eq:u20}
		\lim_{n \to \infty}
		\E
		\left[
		\left|
		D^{(h)}_1 - D^{(h^n)}_1
		\right|
		\right]
		= 0 .
	\end{align}
	In particular, it holds that for any bounded elements $f$ of $\calL^1_G ( \Omega )$,
	\begin{align}
		&\E^h
		\left[
		f(B)
		- \frac{1}{2} \int_0^1 h_s \cdot
		\left(
		d \qv{B}_s \, h_s
		\right)
		\right] \\
		&=
		\lim_{n \to \infty} \E^{h^n}
		\left[
		f(B)
		- \frac{1}{2} \int_0^1 h^n_s \cdot
		\left(
		d \qv{B}_s \, h^n_s
		\right)
		\right] .
		\label{sch}
	\end{align}
\end{lem}

\begin{proof}
	Since $D^{(h)}$ and $D^{(h^n)}$ are $P_{\theta}$-martingales 
	for any $\theta \in \cA$,
	their expectations under $P_{\theta}$ are equal to $1$, and hence
	\begin{align}
		E_{P_{\theta}}
		\left[
		D^{(h)}_1
		\right]
		=
		E_{P_{\theta}}
		\left[
		D^{(h^n)}_1
		\right] .
		\label{sch1}
	\end{align}
	The left-hand side and the right-hand side of \eqref{sch1} are equal to
	\begin{align}
		E_{P_{\theta}}
		\left[
		\left(
		D^{(h)}_1 - D^{(h^n)}_1
		\right)^+
		\right]
		+
		E_{P_{\theta}}
		\left[
		D^{(h)}_1 \wedge D^{(h^n)}_1
		\right]
		\label{sch2}
	\end{align}
	and
	\begin{align}
		E_{P_{\theta}}
		\left[
		\left(
		D^{(h^n)}_1 - D^{(h)}_1
		\right)^+
		\right]
		+
		E_{P_{\theta}}
		\left[
		D^{(h^n)}_1 \wedge D^{(h)}_1
		\right] ,
		\label{sch3}
	\end{align}
	respectively.
	Here $x^+ = x \vee 0$ for $x \in \R$.
	Combining these with \eqref{sch1}, we have the relation
	$E_{P_{\theta}} [( D^{(h)}_1 - D^{(h^n)}_1 )^+]
	= E_{P_{\theta}} [( D^{(h^n)}_1 - D^{(h)}_1 )^+]$, and hence
	\begin{align}
		E_{P_{\theta}}
		\left[
		\left|
		D^{(h)}_1 - D^{(h^n)}_1
		\right|
		\right]
		=
		2 E_{P_{\theta}}
		\left[
		\left(
		D^{(h)}_1 - D^{(h^n)}_1
		\right)^+
		\right] .
	\end{align}
	As $\theta \in \cA$ is arbitrary, it follows that
	\begin{align}
		\E
		\left[
		\left|
		D^{(h)}_1 - D^{(h^n)}_1
		\right|
		\right]
		=
		2 \E
		\left[
		\left(
		D^{(h)}_1 - D^{(h^n)}_1
		\right)^+
		\right] .	
		\label{sch4}
	\end{align}
	By letting $X_n := \log D^{(h^n)}_1 - \log D^{(h)}_1$,
	the right-hand side of \eqref{sch4} is rewritten as
	\begin{align}
		2 \E
		\left[
		\left(
		1- e^{X_n}
		\right)^+
		D^{(h)}_1
		\right],
	\end{align}
	which is bounded from above by
	\begin{align}
		2 \E
		\left[
		\left|
		X_n
		\right|
		D^{(h)}_1
		\right]
		\le
		2 \E
		\left[
		\left|
		X_n
		\right|^2
		\right]^{1/2}
		\E
		\left[
		( D^{(h)}_1 )^2
		\right]^{1/2} .
	\end{align}
	To obtain the lemma, 
	it is  enough to show that $\E [ | X_n |^2 ]$ tends to $0$ as $n \to \infty$ 
	since $\E [ ( D^{(h)}_1 )^2 ]$ is finite by \eqref{eq:u11}.
	Note that
	\begin{align}
		X_n
		=
		\int_0^1
		\left(
		h^n_s - h_s
		\right)
		\cdot d B_s
		- \frac{1}{2} \int_0^1
		\left(
		h^n_s + h_s
		\right)
		\cdot
		\left(
		d \qv{B}_s
		\left(
		h^n_s - h_s
		\right)
		\right) .
	\end{align}
	Since it holds that
	\begin{align}
		\E
		\left[
		\left|
		\int_0^1
		\left(
		h^n_s - h_s
		\right)
		\cdot d B_s
		\right|^2
		\right]
		\le
		\sigma^2_1
		\left\|
		h^n-h
		\right\|^2_{ \MGR }
	\end{align}
	and
	\begin{align}
		&\E
		\left[
		\left|
		\int_0^1
		\left(
		h^n_s + h_s
		\right)
		\cdot
		\left(
		d \qv{B}_s
		\left(
		h^n_s - h_s
		\right)
		\right)
		\right|^2
		\right] \\
		&\le
		\sigma^4_1 \E
		\left[
		\int_0^1
		\left|
		h^n_s + h_s
		\right|^2
		d s \int_0^1
		\left|
		h^n_s - h_s 
		\right|^2
		d s
		\right] \\
		&\le
		4 \sigma^4_1 \sup_{0 \le t \le 1}
		\left\|
		| h_t |
		\right\|^2_{\infty}
		\left\|
		h^n - h
		\right\|^2_{ \MGR } ,
	\end{align}
	we get $\lim_{n \to \infty} \E [| X_n |^2] =0$, and complete the proof.
\end{proof}

Let $f \in \calL^1_G ( \Omega )$ be bounded and 
$\{ f_n \}_{ n \in \N } \subset \bLip ( \Omega )$ such that
\begin{align}
	\| f_n \|_{\infty} \le \| f\|_{\infty}
	\n{ for all } n \in \N
	\qn{and} \qd 
	\lim_{n \to \infty} 
	\E [| f_n (B) - f(B) |] =0 .
	\label{setfn}
\end{align}
For each $f_n$, let $h^n \in ( \MG )^d$ and $A^n \in \calL^2_G ( \Omega )$ 
be as given by \lref{lem:u2} and $P^n \in \calP$ as given by \lref{urem}. 

\begin{lem}\label{uu}
	For each $n \in \N$, we have
	\begin{align}
		\log \E \left[ e^{f_n (B)} \right]
		&= 
		E_{Q^n} 
		\left[
		f_n (B) - \frac{1}{2} \int_0^1 h^n_s \cdot \left( d \qv{B}_s \, h^n_s \right)
		\right] ,
		\label{u}
	\end{align}
	where $Q^n := D^{(h^n)}_1 P^n$.
\end{lem}

\begin{proof}
	Setting $\ol{B} := T^{-h^n} (B)$, we have by \lref{lem:u2}, 
	\begin{align}
		\log \E \left[ e^{f_n (B) } \right] 
		=
		f_n (B)
		-
		\int_0^1 h^n_s \cdot d \ol{B}_s
		- 
		\frac{1}{2} \int_0^1 h^n_s \cdot \left( d \qv{B}_s \, h^n_s \right) 
		+ A^n .
	\end{align}
	Since $h^n$ is bounded, we have by (ii) of \lref{l:sym}, 
	\begin{align}
		\E 
		\left[
		\left(
		\int_0^1 h^n_s \cdot d \ol{B}_s
		\right)
		D^{(h^n)}_1
		\right]
		=
		- \E 
		\left[
		-
		\left(
		\int_0^1 h^n_s \cdot d \ol{B}_s
		\right)
		D^{(h^n)}_1
		\right]
		= 0.
	\end{align}
	By \eqref{;char3}, this relation holds with $\E$ replaced by $\hat{\E}$, 
	which results in 
	\begin{align}
		E_P 
		\left[
		\left(
		\int_0^1 h^n_s \cdot d \ol{B}_s
		\right)
		D^{(h^n)}_1
		\right]
		= 0
		\qn{for all } P \in \calP .
	\end{align}
	Combining this with $P^n ( A^n =0 ) =1$ leads to \eqref{u}. 
\end{proof}

Now we are in a position to prove the upper bound in \tref{;tmain}.

\begin{prop}\label{prop:u}
	\eqref{Ubound} holds for any bounded elements $f$ of $\calL^1_G ( \Omega )$.
\end{prop}

\begin{proof}
	For a bounded $f \in \calL^1_G ( \Omega )$, 
	let $\{ f_n \}_{n \in \N} \subset \bLip ( \Omega )$ satisfy \eqref{setfn}, 
	and let $h^n$, $A^n$ and $P^n$ be as above. 
	Then for each $n \in \N$, we have by \lref{uu},
	\begin{align}
		\log \E \left[ e^{f(B)} \right] 
		&=
		\log \E \left[ e^{f(B)} \right]
		-
		\log \E \left[ e^{ f_n (B) } \right]
		+
		E_{Q^n} [ f_n (B) - f(B) ] \\
		&\qd
		+ E_{Q^n} 
		\left[
		f (B) - \frac{1}{2} \int_0^1 h^n_s \cdot \left( d \qv{B}_s \, h^n_s \right)
		\right] . 
		\label{eq:fn}
	\end{align}
	
	For the difference of the first two terms in the right-hand side of \eqref{eq:fn},
	as seen in \eqref{expLip}, we have
	\begin{align}
		\left|
		\log \E \left[ e^{f(B)} \right] 
		- 
		\log \E \left[ e^{f_n (B)} \right]
		\right|
		\to 0 
		\qn{as } n \to \infty . 
	\end{align}
	As to the third term in the right-hand side of \eqref{eq:fn}, 
	we see from Lemmas~\ref{lem:u2} and \ref{urem} that
	\begin{align}
		D^{(h^n)}_1
		=
		\frac{ e^{ f_n (B) } }{ \E \left[ e^{f_n (B) } \right] }
		\le 
		e^{2 \| f \|_{\infty} } 
		\qd P^n \n{-a.s.},
	\end{align}
	and hence 
	\begin{align}
		E_{Q^n} [ f_n (B) - f(B) ]
		&\le 
		e^{2 \| f \|_{\infty}}
		E_{P^n} [| f_n (B) - f(B) |] \\
		&\le
		e^{2 \| f \|_{\infty}}
		\hat{\E} [| f_n (B) - f(B) |] ,
	\end{align}
	which converges to $0$ as $n \to \infty$ 
	by \eqref{;char3} and \eqref{setfn}.
	
	Therefore it remains to estimate the last term 
	in the right-hand side of \eqref{eq:fn}.
	For this purpose, we first observe the bound 
	\begin{align}
		E_{Q^n} 
		\left[
		f(B) 
		- 
		\frac{1}{2} \int _0^1 
		h^n_s \cdot ( d \qv{B}_s \, h^n_s )
		\right]
		&\le 
		\hat{\E}
		\left[
		\left(
		f(B) 
		- 
		\frac{1}{2} \int _0^1 
		h^n_s \cdot ( d \qv{B}_s \, h^n_s )
		\right)
		D^{(h^n)}_1
		\right] \\
		&=
		\E ^{h^n} 
		\left[
		f(B) 
		- 
		\frac{1}{2} \int _0^1 
		h^n_s \cdot ( d \qv{B}_s \, h^n_s )
		\right] ,
		\label{EQn}
	\end{align}
	where the equality follows from \eqref{;char3}.
	Fix $\ve > 0$ arbitrarily. 
	From \eqref{sch} in \lref{lem:sche}, we see that,
	for every $n \in \N$, there exists an $h \equiv h^{(n, \ve)} \in ( \sbl )^d$ 
	such that 
	\begin{align}\label{eq:u10}
		\hspace*{-5mm}
		\E^{h^n} 
		\left[ 
		f(B) - \frac{1}{2} \int_0^1 h^n_s \cdot 
		\left( d \qv{B}_s \, h^n_s \right) 
		\right]
		&\le 
		\E^h
		\left[ 
		f(B) - \frac{1}{2} \int_0^1 h_s \cdot 
		\left( d \qv{B}_s \, h_s \right) 
		\right]
		+ \ve .
	\end{align}
	Let $h \in ( \sbl )^d$ be written as \eqref{sbl} 
	and define a simple process $\hat{h}$ as follows:
	\begin{align}
		&\hat{\xi}_0 := \xi_0 , 
		& &
		\hat{h}_t := \hat{\xi}_0 ,~ 
		t_0 \le t < t_1, \\
		&\hat{\xi}_1 := \xi_1 
		\left( B_t + \int_0^t d \qv{B}_s \, \hat{h}_s ,~ t \le t_1 \right) ,
		& &
		\hat{h}_t := \hat{\xi}_1 ,~ 
		t_1 \le t < t_2, \\
		&\hspace{15em} \vdots & & \\
		&\hat{\xi}_{m-1} := \xi_{m-1} 
		\left( B_t + \int_0^t d \qv{B}_s \, \hat{h}_s ,~ t \le t_{m-1} \right) ,
		& &
		\hat{h}_t := \hat{\xi}_{m-1} ,~ 
		t_{m-1} \le t < t_m .
	\end{align}
	By this construction, it is obvious that $\hat{h}$ is in 
	$( M^{2,0}_G (0,1) )^d$ and bounded. 
	Furthermore, it can be shown inductively that 
	\begin{align}
		h_t
		= \hat{h}_t ( T^{-h} (B) ),
		\qd t_k \le t < t_{k+1}, 
		\label{hath}
	\end{align}
	for all $k=0,1 \ddd m-1$. 
	Put $\ol{B} := T^{-h} (B)$.
	By \tref{;gir}, we have
	\begin{align}
		&\E^h
		\left[ 
		f(B) - \frac{1}{2} \int_0^1 h_s \cdot 
		\left( d \qv{B}_s \, h_s \right) 
		\right] \\
		&= 
		\E^h 
		\left[
		f \left( \ol{B} + \int_0^{\cdot} d \qv{ \ol{B}}_s \, \hat{h}_s (\ol{B}) \right)
		- \frac{1}{2} \int_0^1 \hat{h}_s (\ol{B}) \cdot 
		\left( d \qv{\ol{B}}_s \, \hat{h}_s (\ol{B}) \right) 
		\right] \\
		&= 
		\E 
		\left[
		f \left( B + \int_0^{\cdot} d \qv{B}_s \, \hat{h}_s \right)
		- \frac{1}{2} \int_0^1 \hat{h}_s \cdot 
		\left( d \qv{B}_s \, \hat{h}_s \right) 
		\right] ;
	\end{align}
	for the validity of the second equality, see \rref{rem:gir}.
	Combining this with \eqref{EQn} and \eqref{eq:u10}, we obtain
	\begin{align}
		&E_{Q^n}
		\left[ 
		f(B) - \frac{1}{2} \int_0^1 h^n_s \cdot \left( d \qv{B}_s \, h^n_s \right) 
		\right] \\
		&\le \sup_{h \in ( \MG )^d} \E \left[
		f \left( B + \int_0^{\cdot} d \qv{B}_s \, h_s \right)
		- \frac{1}{2} \int_0^1 h_s \cdot \left( d \qv{B}_s \, h_s \right) \right]
		+ \ve \label{eq:l20}
	\end{align}
	for all $n \in \N$, and complete the proof.
\end{proof}

%%%%%%%%%%%%%%%%%%%%%%%%%%%%%%%%%%%%%%%%%%%%%%%%%%%%%%%%%%%%%%%%%%%%%
%%%%%%%%%%%%%%%%%%%%%%%%%%%%%%%%%%%%%%%%%%%%%%%%%%%%%%%%%%%%%%%%%%%%%
%%%%%%%%%%%%%%%%%%%%%%%%%%%%%%%%%%%%%%%%%%%%%%%%%%%%%%%%%%%%%%%%%%%%%
%%%%%%%%%%%%%%%%%%%%%%%%%%%%%%%%%%%%%%%%%%%%%%%%%%%%%%%%%%%%%%%%%%%%%
%%%%%%%%%%%%%%%%%%%%%%%%%%%%%%%%%%%%%%%%%%%%%%%%%%%%%%%%%%%%%%%%%%%%%
\section{An application to large deviations}\label{;ld}

In this section, we apply \tref{;tmain} 
to the derivation of the Laplace principles 
for $\{ ( \sqrt{\ve} B , \qv{B} ) ; \ve >0 \}$ 
and $\{ \sqrt{\ve} B ; \ve >0 \}$.
Similarly to the classical case, 
the Laplace principle implies the large deviation principle, 
and hence we recover the large deviation principles for these two families,
which are originally proved by Gao-Jiang \cite{Gao;10} 
through discretization technique.

First we formulate the Laplace principle under $G$-expectation as follows:
Let $\{ X^{\ve} ; \ve > 0 \}$ be a family of random variables 
taking values in a Polish space $\calX$.
We let $I$ be a rate function on $\calX$, that is, 
a mapping from $\calX$ into $[0, \infty ]$ 
such that for each $M>0$ the revel set $\{ x \in \calX : I(x) \le M \}$ 
is a compact subset of $\calX$.
We say that $\{ X^{\ve} ; \ve >0 \}$ satisfies the Laplace principle on $\calX$ 
with rate function $I$ if for all bounded 
and continuous functions $\Phi : \calX \to \R$, it holds that
\begin{align}\label{eq:L}
	\lim_{\ve \to 0} \ve \log \E \left[ \exp
	\left( \frac{ \Phi \left( X^{\ve} \right) }{\ve} \right) \right]
	= \sup_{x \in \calX} \left\{ \Phi (x) - I (x) \right\} .
\end{align}

The following proposition can be proved 
through the same argument as in the classical case 
(see, e.g., the proof of \cite[Theorem~1.2.3]{Dupuis;97}).
We remark that the validity of the converse assertion 
is ensured by Lemma~A.3 in \cite{Gao;10}.

\begin{prop}\label{prop:L}
	The Laplace principle implies the large deviation principle 
	with the same rate function.
	More precisely,
	if the Laplace limit \eqref{eq:L} holds 
	for all bounded continuous functions $\Phi : \calX \to \R$, 
	then $\{ X^{\ve} ; \ve >0 \}$ satisfies 
	the large deviation principle on $\calX$ with rate function $I$.
\end{prop}

\begin{rem}\label{rem:L}
	As in the classical case, it is enough to check that 
	the Laplace limit \eqref{eq:L} is valid 
	for all bounded Lipschitz continuous functions $\Phi$ 
	in order to see $\{ X^{\ve} ; \ve >0 \}$ satisfies the Laplace principle.
\end{rem}

Denote by $C ( [0,1] ; \R^{d \times d} )$ (resp. by $C ( [0,1] ; \R^d )$) 
the space of all $\R^{d \times d}$-valued (resp. $\R^d$-valued) continuous functions 
on $[0,1]$ vanishing at 0. 
We equip $C ( [0,1] ; \R^{d \times d} )$ with the distance
\begin{align}
	\n{dist} (y^1 , y^2) 
	:= 
	\sup_{0 \le t \le 1} 
	\left\| y^1 (t) - y^2 (t) \right\| ,
	\qd y^1 , y^2 \in C ( [0,1] ; \R^{d \times d} ) ,
\end{align}
where $\| A \| = \sqrt{ \tr [A A^*] }$ for $A \in \R^{d \times d}$, 
and equip $C ( [0,1] ; \R^d )$ with the distance $\rho$ defined by \eqref{eq:rho}.
Set
\begin{align}
	&\bH 
	:= 
	\left\{ 
	x \in C ( [0,1] ; \R^d ) : 
	x \n{ is absolutely continuous and }
	\int_0^1 | \dot{x} \left( t \right) |^2 \, d t
	< \infty 
	\right\} , \\
	&\A 
	:= 
	\left\{
	\begin{aligned}
		y \in C ( [0,1] ; \R^{d \times d} ) :~ 
		&y \n{ is absolutely continuous and } \\
		&\dot{y} (t)
		\in \{ \gamma \gamma^* : \gamma \in \Theta \} 
		\n{ for a.e. } t \in [0,1]
	\end{aligned}
	\right\} .
\end{align}
Here $\dot{x}$ and $\dot{y}$ denote 
the derivatives $dx/dt$ and $dy/dt$, respectively, provided that they exist. 
We define rate functions 
$J : C ( [0,1] ; \R^d ) \times C ( [0,1] ; \R^{d \times d} ) \to [0, \infty ]$ 
and $I : C ( [0,1] ; \R^d ) \to [0, \infty ]$ by
\begin{align}
	&J (x,y) = 
	\left\{
	\begin{aligned}
		&\frac{1}{2} \int_0^1 \dot{x} (t)
		\cdot \left( \dot{y} (t)^{-1} \dot{x} (t) \right) d t
		& &\n{if } (x,y) \in \bH \times \A , \\
		&+ \infty 
		& &\n{otherwise,}
	\end{aligned}
	\right. \\
	&I (x) = 
	\left\{
	\begin{aligned}
		&\frac{1}{2} \int_0^1 \inf_{\gamma \in \Theta}
		| \gamma^{-1} \dot{x} (t) |^2 \, d t
		& &\n{if } x \in \bH , \\
		&+ \infty 
		& &\n{otherwise.}
	\end{aligned}
	\right.
\end{align}
In the following, we abbreviate the notation as $\sup_{x}$ (resp. $\sup_{(x,y)}$) 
when we take the supremum over $x \in C([0,1] ; \R^d)$ 
(resp. $(x,y) \in C ([0,1] ; \R^d) \times C([0,1] ; \R^{d \times d})$). 

\begin{prop}\label{prop:apl}
	$(\mathrm{i})$ 
	For any bounded Lipschitz continuous function $\Psi : C ( [0,1] ; \R^d )$
	$ \times C ( [0,1] ; \R^{d \times d} ) \to \R$,
	\begin{align}\label{eq:BB}
		\begin{aligned}
			\lim_{\ve \to 0} 
			\ve \log \E 
			\left[
			\exp \left( \frac{ \Psi ( \sqrt{\ve} B , \qv{B} ) }{\ve} \right) 
			\right]
			&= \sup_{(x,y)} 
			\left\{ \Psi (x,y) - J (x,y) \right\} .
		\end{aligned}
	\end{align}
	$(\mathrm{ii})$ 
	For any bounded Lipschitz continuous function $\Phi : C ( [0,1] ; \R^d ) \to \R$,
	\begin{align}\label{eq:B}
		\lim_{\ve \to 0} 
		\ve \log \E 
		\left[
		\exp \left( \frac{ \Phi ( \sqrt{\ve} B ) }{\ve} \right) 
		\right]
		&= \sup_{x} 
		\left\{ \Phi (x) - I (x) \right\} .
	\end{align}
\end{prop}

From \pref{prop:L} and \rref{rem:L}, we see that \pref{prop:apl} implies that 
the family $\{ ( \sqrt{\ve} B , \qv{B} ) ; \ve >0 \}$ 
(resp.\ $\{ \sqrt{\ve} B ; \ve >0 \}$) satisfies the large deviation principle on 
$C ( [0,1] ; \R^d ) \times C ( [0,1] ; \R^{d \times d} )$ (resp.\ $C ( [0,1] ; \R^d )$) 
with rate function $J$ (resp.\ $I$).

We begin with a lemma which is an application of \tref{;tmain}.

\begin{lem}\label{lem:apl}
	For every $\ve>0$, we have the following.
	
	$(\mathrm{i})$ 
	For any bounded Lipschitz continuous function 
	$\Psi : C ( [0,1] ; \R^d ) \times C ( [0,1] ; \R^{d \times d} )$
	$\to \R$,
	\begin{align}
		&\ve \log \E 
		\left[ 
		\exp \left( \frac{\Psi ( \sqrt{\ve} B , \qv{B} )}{ \ve } \right) 
		\right] \\
		&= \sup_{h \in ( \MG )^d}
		\E 
		\left[ 
		\Psi \left( \sqrt{\ve} B + \int_0^{\cdot} d \qv{B}_s \, h_s , \qv{B} \right)
		- \frac{1}{2} \int_0^1 h_s \cdot \left( d \qv{B}_s \, h_s \right)
		\right] .
		\label{eq:a1}
	\end{align}
	
	$(\mathrm{ii})$ 
	For any bounded and Lipschitz continuous function $\Phi : C ( [0,1] ; \R^d ) \to \R$,
	\begin{align}\label{eq:a2}
		&\ve \log \E 
		\left[ 
		\exp \left( \frac{\Phi ( \sqrt{\ve} B )}{ \ve } \right) 
		\right] \\
		&= \sup_{h \in ( \MG )^d}
		\E 
		\left[ 
		\Phi \left( \sqrt{\ve} B + \int_0^{\cdot} d \qv{B}_s \, h_s \right)
		- \frac{1}{2} \int_0^1 h_s \cdot \left( d \qv{B}_s \, h_s \right)
		\right] .
	\end{align}
\end{lem}

\begin{proof}
	The proofs of (i) and (ii) are similar, so we only show (i).
	
	We first check that the functional $\Psi (\sqrt{\ve} B , \qv{B})$
	is a bounded element of $\calL^1_G ( \Omega )$. 
	For $n \in \N$, let $\Delta_n = \{ 0= t_0 < t_1 < \dots < t_n =1 \}$ 
	be the partition of $[0,1]$ such that $t_k - t_{k-1} = 1/n$ for all $k= 1 \ddd n$.
	For every $y \in C ([0,1] ; \R^{d \times d})$ we denote by $( y )^{\Delta_n}$ 
	the polygonal approximation of $y$ such that 
	$( y )^{\Delta_n} (t_k) = y (t_k)$, $k=1 \ddd n$. 
	Since $\qv{B}_{t} \in ( \calL^1_G ( \Omega ) )^{d \times d} $ 
	for each $t \in [0,1]$, the mapping
	\begin{align}
		\Omega \ni \omega \mapsto
		\Psi_n  (\omega)
		:= 
		\Psi 
		\left( \sqrt{\ve} \omega ,  \left( \qv{B} \right)^{\Delta_n} (\omega) 
		\right)
	\end{align}
	has a q.c.\ version.
	$\Psi_n$ is also bounded, and hence is in $\calL^1_G ( \Omega )$.
	By the Lipschitz continuity of $\Psi$, we have
	\begin{align}
		\ol{\E} 
		\left[
		\left| 
		\Psi \left( \sqrt{\ve} B , \qv{B} \right) 
		-
		\Psi_n \left( B \right)
		\right|
		\right]
		&\le 
		\lip ( \Psi ) 
		\ol{\E}
		\left[ 
		\n{dist}
		\left( 
		\qv{B} , ( \qv{B} ) ^{\Delta_n}
		\right)
		\right] .
		\label{poly}
	\end{align}
	Note that
	\begin{align}
		&\n{dist}
		\left( 
		\qv{B} , ( \qv{B} ) ^{\Delta_n}
		\right) \\
		&= \max _{1 \le k \le n}
		\sup_{t_{k-1} \le t \le t_k}
		\left\| 
		\qv{B} _t - \qv{B} _{t_{k-1}} - n ( t-t_{k-1} ) ( \qv{B} _{t_k} - \qv{B} _{t_{k-1}} )
		\right\| \\
		&\le 2 \max _{1 \le k \le n}
		\sup_{t_{k-1} \le t \le t_{k}}
		\left\| 
		\qv{B} _t - \qv{B} _{t_{k-1}}
		\right\| , \label{eq:dist}
	\end{align}
	and that
	$\left\| \qv{B}_t - \qv{B}_s \right\| \le d \sigma_1^2 |t-s|$ 
	for all $0 \le s,t \le 1$ q.s.
	Then the right-hand side of \eqref{poly} is estimated from above by 
	\begin{align}
		\frac{2 d \lip ( \Psi ) \sigma_1^2 }{n} ,
	\end{align}
	which tends to $0$ as $n \to \infty$,
	and hence $\Psi \left( \sqrt{\ve} B , \qv{B} \right) \in \calL^1_G ( \Omega )$.
	
	Now let us verify \eqref{eq:a1}.
	By \tref{;tmain}, we have
	\begin{align}
		&\ve \log \E 
		\left[ 
		\exp \left( \frac{ \Psi (\sqrt{\ve} B , \qv{B} ) }{ \ve } \right) 
		\right] \\
		&= 
		\sup_{h \in ( \MG )^d }
		\E 
		\left[ 
		\Psi \left( \sqrt{\ve} T^h (B) , \qv{T^h (B)} \right)
		- \frac{\ve}{2} \int_0^1 h_s \cdot \left( d \qv{B}_s \, h_s \right) 
		\right] \\
		&= 
		\sup_{h \in ( \MG )^d }
		\E 
		\left[
		\Psi 
		\left( \sqrt{\ve} B + \int_0^{\cdot} d \qv{B}_s \, \sqrt{\ve} h_s , \qv{B} \right)
		- \frac{1}{2} \int_0^1 
		\sqrt{\ve} h_s \cdot \left( d \qv{B}_s \, \sqrt{\ve} h_s \right)
		\right] \\
		&= 
		\sup_{h \in ( \MG )^d }
		\E 
		\left[
		\Psi \left( \sqrt{\ve} B + \int_0^{\cdot} d \qv{B}_s \, h_s , \qv{B} \right)
		- \frac{1}{2} \int_0^1 h_s \cdot \left( d \qv{B}_s \, h_s \right)
		\right],
	\end{align}
	which is \eqref{eq:a1}.
\end{proof}

By using \lref{lem:apl}, we prove \pref{prop:apl}.

\begin{proof}[Proof of \pref{prop:apl}]
	(i) By the Lipschitz continuity of $\Psi$, we have
	\begin{align}
		&\begin{aligned}
			\left| 
			\sup_{h \in ( \MG )^d }
			\E 
			\left[ 
			\Psi \left( \sqrt{\ve} B + \int_0^{\cdot} d \qv{B}_s \, h_s , \qv{B} \right)
			- \frac{1}{2} \int_0^1 h_s \cdot \left( d \qv{B}_s \, h_s \right)
			\right] 
			\right. \\
			\left. 
			- \sup_{h \in ( \MG )^d }
			\E 
			\left[ 
			\Psi \left( \int_0^{\cdot} d \qv{B}_s \, h_s , \qv{B} \right)
			- \frac{1}{2} \int_0^1 h_s \cdot \left( d \qv{B}_s \, h_s \right)
			\right] 
			\right|
		\end{aligned} \\
		&\le 
		\sup_{h \in ( \MG )^d} 
		\E 
		\left[ 
		\left|
		\Psi \left( \sqrt{\ve} B + \int_0^{\cdot} d \qv{B}_s \, h_s , \qv{B} \right)
		- \Psi \left( \int_0^{\cdot} d \qv{B}_s \, h_s , \qv{B} \right) 
		\right| 
		\right] \\
		&\le 
		C \sqrt{\ve} . \label{eq:a3}
	\end{align}
	Here $C := \lip ( \Psi ) \ol{\E} \left[ \sup_{0 \le t \le 1} | B_t | \right]$,
	whose finiteness follows from the Cauchy-Schwarz inequality and Doob's inequality: 
	\begin{align}
		\ol{\E} 
		\left[ \sup_{0 \le t \le 1} | B_t | \right]
		&\le
		\ol{\E} 
		\left[ \sup_{0 \le t \le 1} | B_t |^2 \right] ^{1/2} \\
		&= 
		\sup_{\theta \in \cA} 
		E_P 
		\left[ 
		\sup_{0 \le t \le 1} 
		\left| \int_0^t \theta _s \, d W _s \right| ^2 
		\right] ^{1/2} \\
		&\le
		2 \sup_{\theta \in \cA} 
		E_P 
		\left[
		\int_0^1 \tr [ \theta _s \theta _s^* ] \, d s
		\right] ^{1/2}
		\le 2 \sqrt{d} \sigma _1 ;
		\label{C}
	\end{align}
	for the equality, 
	recall that, in the definition of the upper expectation $\ol{\E}$, 
	the supremum is taken over the laws of 
	$\int _0^{\cdot} \theta _s \, d W_s$, $\theta \in \cA$, 
	under the probability measure $P$.
	Combining \lref{lem:apl} with \eqref{eq:a3}, 
	we see that
	\begin{align}
		&\lim_{\ve \to 0} 
		\ve\log \E 
		\left[
		\exp \left( \frac{ \Psi ( \sqrt{\ve} B , \qv{B} ) }{\ve} \right) 
		\right] \\
		&= 
		\sup_{h \in ( \MG )^d }
		\E 
		\left[ 
		\Psi \left( \int_0^{\cdot} d \qv{B}_s \, h_s , \qv{B} \right)
		- \frac{1}{2} \int_0^1 h_s \cdot \left( d \qv{B}_s \, h_s \right)
		\right] .
	\end{align}
	Hence what to show is that
	\begin{align}\label{eq:a4}
		&\sup_{h \in ( \MG )^d }
		\E 
		\left[ 
		\Psi \left( \int_0^{\cdot} d \qv{B}_s \, h_s , \qv{B} \right)
		- \frac{1}{2} \int_0^1 h_s \cdot \left( d \qv{B}_s \, h_s \right)
		\right] \\
		&= 
		\sup_{(x,y)} 
		\left\{ \Psi (x,y) - J (x,y) \right\} .
	\end{align}

	First we prove the upper bound.
	For $\theta \in \cA$ and $h \in ( \MG )^d$ , we set
	\begin{align}
		X_t 
		\equiv X_t^{( \theta , h )} 
		:= 
		\int_0^t \theta_s \theta_s^* h^{(\theta)}_s \, d s ,
		\qd
		Y_t 
		\equiv Y_t^{( \theta )}
		:= 
		\int_0^t \theta_s \theta_s^* \, d s
		\qn{for } 0 \le t \le 1 .
	\end{align}
	Here $h^{(\theta)}$ is defined by
	\begin{align}
		h^{(\theta)}_t 
		= 
		h_t \left( \int_0^{\cdot} \theta_s \, d W_s \right) 
		\qn{for } 0 \le t \le 1 .
	\end{align}
	Then $X \in \bH$ and $Y \in \A$ $P$-a.s., 
	and hence
	\begin{align}
		E_{P_{\theta}} 
		\left[ 
		\Psi \left( \int_0^{\cdot} d \qv{B}_s \, h_s , \qv{B} \right)
		- \frac{1}{2} \int_0^1 h_s \cdot \left( d \qv{B}_s \, h_s \right) 
		\right]
		&= 
		E_P 
		\left[ \Psi (X,Y) - J (X,Y) \right] \\
		&\le 
		\sup_{(x,y)} 
		\left\{ \Psi (x,y) - J (x,y) \right\} .
	\end{align}
	Since $\theta \in \cA$ and $h \in ( \MG )^d$ 
	are arbitrary, we get the upper bound in \eqref{eq:a4}.

	Next we prove the lower bound in \eqref{eq:a4}.
	If $x \not\in \bH$ or $y \not \in \A$, 
	the right-hand side of \eqref{eq:a4} is $- \infty$, 
	so we take an arbitrary $(x,y) \in \bH \times \A$.
	Let $y$ be written as
	\begin{align}
		y(t) 
		= 
		\int_0^t g(s) g(s)^* \, d s ,
		\qd 0 \le t \le 1,
	\end{align}
	for some deterministic measurable function $g : [0,1] \to \Theta$.
	We denote by $P_g$ the law of $\int_0^{\cdot} g (s) \, d W_s$ 
	and define a deterministic function $\eta$ by
	\begin{align}
		\eta_t
		= 
		\dot{y} (t)^{-1} \dot{x} (t)
		\qn{for a.e.\ } t \in [0,1] .
	\end{align}
	Then $\eta$ is in $( \MG )^d$, whence
	\begin{align}
		&\sup_{h \in ( \MG )^d }
		\E 
		\left[ 
		\Psi \left( \int_0^{\cdot} d \qv{B}_s \, h_s , \qv{B} \right)
		- \frac{1}{2} \int_0^1 h_s \cdot \left( d \qv{B}_s \, h_s \right)
		\right] \\
		&\ge 
		E_{P_g} 
		\left[ 
		\Psi \left( \int_0^{\cdot} d \qv{B}_s \, \eta_s , \qv{B} \right)
		- \frac{1}{2} \int_0^1 \eta_s \cdot \left( d \qv{B}_s \, \eta_s \right) 
		\right] \\
		&= 
		E_P 
		\left[ 
		\Psi \left( \int_0^{\cdot} \dot{y} (s) \eta_s \, d s , y  \right)
		- \frac{1}{2} \int_0^1 \eta_s \cdot \left( \dot{y} (s) \eta_s \right) d s 
		\right] \\
		&=
		\Psi (x,y) -J (x,y).
	\end{align}
	Taking the supremum of the right-hand side, we obtain the lower bound.

	(ii) Similarly to (i), we have by \lref{lem:apl} 
	\begin{align}
		&\lim_{\ve \to 0} 
		\ve \log \E 
		\left[
		\exp \left( \frac{ \Phi ( \sqrt{\ve} B ) }{\ve} \right) 
		\right] \\
		&= 
		\sup_{h \in ( \MG )^d}
		\E 
		\left[ 
		\Phi \left( \int_0^{\cdot} d \qv{B}_s \, h_s \right)
		-\frac{1}{2} \int_0^1 h_s \cdot \left( d \qv{B}_s \, h_s \right) 
		\right] .
	\end{align}
	Therefore it is sufficient to show
	\begin{align}\label{eq:a5}
	\hspace{-8mm}
		\sup_{h \in ( \MG )^d}
		\E 
		\left[ 
		\Phi \left( \int_0^{\cdot} d \qv{B}_s \, h_s \right)
		-\frac{1}{2} \int_0^1 h_s \cdot \left( d \qv{B}_s \, h_s \right) 
		\right]
		= 
		\sup_{x} 
		\left\{ \Phi (x) - I (x) \right\} .
	\end{align}

	For the upper bound, take $\theta \in \cA$ and $h \in ( \MG )^d$
	arbitrarily and let $X \equiv X^{(\theta ,h)}$ be as in the proof of (i).
	Then we have
	\begin{align}
		E_{P_{\theta}} 
		\left[ 
		\Phi \left( \int_0^{\cdot} d \qv{B}_s \, h_s \right)
		-\frac{1}{2} \int_0^1 h_s \cdot \left( d \qv{B}_s \, h_s \right) 
		\right]
		&= 
		E_P 
		\left[ 
		\Phi (X) 
		- \frac{1}{2} \int_0^1 \left| \theta_s^{-1} \dot{X}_s \right|^2 d s 
		\right] \\
		&\le  
		E_P 
		\left[ \Phi (X) - I (X) \right] \\
		&\le 
		\sup_{x} 
		\left\{ \Phi (x) - I (x) \right\} .
	\end{align}

	For the lower bound in \eqref{eq:a5}, we fix any $x \in \bH$.
	By the measurable selection (see, e.g., \cite[Lemma~A.1]{Boue;98}), 
	there exists a measurable map $\Gamma : \R^d \to \Theta$ such that
	\begin{align}
		\left| \Gamma ( \xi )^{-1} \xi \right|
		= 
		\inf_{\gamma \in \Theta} 
		\left| \gamma^{-1} \xi \right|
		\qn{for all } \xi \in \R^d .
	\end{align}
	For such $\Gamma$, 
	define $g (s) := \Gamma \left( \dot{x} (s) \right)$ 
	for a.e.\ $s \in [0,1]$,
	and note that
	\begin{align}
		I (x)
		= 
		\frac{1}{2} \int_0^1 \left| g(s)^{-1} \dot{x} (s) \right|^2 d s.
	\end{align}
	Denote by $P_g$ the law of $\int_0^{\cdot} g(s) \, d W_s$ and set
	\begin{align}
		z(t)
		:= 
		\int_0^t \left( g(s) g(s)^* \right)^{-1} \dot{x} (s) \, d s.
	\end{align}
	Since $\dot{z} \in ( \MG )^d$, it follows that
	\begin{align}
		&\sup_{h \in ( \MG )^d}
		\E 
		\left[ 
		\Phi \left( \int_0^{\cdot} d \qv{B}_s \, h_s \right)
		-\frac{1}{2} \int_0^1 h_s \cdot \left( d \qv{B}_s \, h_s \right) 
		\right] \\
		&\ge E_{P_g}
		\left[ 
		\Phi \left( \int_0^{\cdot} d \qv{B}_s \, \dot{z} (s) \right)
		- \frac{1}{2} \int_0^1  \dot{z} (s) \cdot \left( d \qv{B}_s \, \dot{z} (s) \right) 
		\right] \\
		&= 
		\Phi (x)
		- \frac{1}{2} \int_0^1 \left| g(s)^{-1} \dot{x} (s) \right|^2 \, d s \\
		&= 
		\Phi (x) - I (x) ,
	\end{align}
	and hence by taking the supremum of the rightmost side, 
	we obtain the lower bound in \eqref{eq:a5}.
\end{proof}

\begin{rem}
	Let $\calE [X] := - \E [-X]$ for $X \in \calL^1_G (\Omega)$.
	By similar arguments to those in Section~\ref{;pr2}, 
	we can also derive a variational representation
	\begin{align}
		\log \calE
		\left[
		e^{f (B)}
		\right]
		= \sup_{ h \in ( \MG )^d } \calE 
		\left[
		f \left( B  + \int_0^{\cdot} d \qv{B}_s \, h_s \right ) 
		- \frac{1}{2} \int_0^1 h_s \cdot \left( d \qv{B}_s \, h_s \right)
		\right]
		\label{lvar}
	\end{align}
	for any bounded elements $f$ of $\calL^1_G ( \Omega )$.
	Note that by \eqref{;char2}, 
	$\calE [X] = \inf_{\theta \in \cA} E_{ P_{\theta} } [X]$,
	to which we may associate the ``lower" capacity $\underline{c}$ via
	\begin{align}
		\underline{c} \left( A \right)
		= \inf_{\theta \in \cA} P_{\theta} (A)
		\qn{for } A \in \calB ( \Omega ) .
	\end{align}
	We expect that similar applications to those given in this section 
	are also possible under $\calE$ and $\underline{c}$.
	A key to the validity of \eqref{lvar} is the identity \eqref{sch} in \lref{lem:sche}.
\end{rem}

%%%%%%%%%%%%%%%%%%%%%%%%%%%%%%%%%%%%%%%%%%%%%%%%%%%%%%%%%%%%%%%%%%%%%
%%%%%%%%%%%%%%%%%%%%%%%%%%%%%%%%%%%%%%%%%%%%%%%%%%%%%%%%%%%%%%%%%%%%%
%%%%%%%%%%%%%%%%%%%%%%%%%%%%%%%%%%%%%%%%%%%%%%%%%%%%%%%%%%%%%%%%%%%%%
%%%%%%%%%%%%%%%%%%%%%%%%%%%%%%%%%%%%%%%%%%%%%%%%%%%%%%%%%%%%%%%%%%%%%
%%%%%%%%%%%%%%%%%%%%%%%%%%%%%%%%%%%%%%%%%%%%%%%%%%%%%%%%%%%%%%%%%%%%%
\section{An absolute continuity result}\label{s:rem}

We conclude this paper with the proof of 
an absolute continuity relationship between 
$B$ and $B + \int_0^{\cdot} d \qv{B}_s \, h_s$ under the capacity, 
which plays a key role in Subsection~\ref{;pr1} and is of independent interest itself. 
Let $c$ and $\hat{c}$ be the capacities 
given in \eqref{c} and \eqref{hatc}, respectively, 
and $T^h (B)$ for $h \in ( \MG )^d$ as in Section~\ref{;pr2}. 

\begin{prop}\label{;abs}
	Let $h \in ( \MG )^d$ be arbitrary. 
	Then for any $\ve > 0$, there exists $\delta >0$ such that
	\begin{align}
		c \left( T^h (B) \in A \right)
		< \ve 
		\qn{for all } A \in \calB ( \Omega )
		\n{ with } \hat{c} \left( A \right) < \delta .
		\label{abs1}
	\end{align}
	In particular, 
	\begin{align}
		c \left( T^h (B) \in N \right)
		= 0 
		\qn{for all } N \in \calB ( \Omega )
		\n{ with } c \left( N \right) = 0 .
		\label{abs2}
	\end{align}
\end{prop}

\begin{rem}
	Note that, since the reverse Fatou lemma does not hold in full generality under $c$, 
	we cannot conclude from \eqref{abs2} the stronger result than \eqref{abs1}:
	\begin{align}
		c \left( T^h (B) \in A \right) < \ve
		\qn{for all } A \in \calB ( \Omega ) 
		\n{ with } c \left( A \right) < \delta .
	\end{align}
\end{rem}

We start with the following lemma. 

\begin{lem}\label{;appli}
	Let $h \in ( \sbl )^d$. It holds that 
	\begin{align}
		\E ^{\OL{h}}
		\left[ 
		\log D^{( \OL{h} )}_1
		\right] 
		\le 
		\frac{1}{2} \sigma_1 ^2 \| h \|_{ \MGR }^2 .
	\end{align}
	Here $\OL{h}$ is the associated bounded element of $( M^{2,0}_G (0,1) )^d$ 
	defined so that the relation \eqref{;barh} holds. 
\end{lem}

\begin{proof}
	First observe that by the boundedness of $h$, 
	$( \log D^{( \OL{h} )}_1 ) D^{( \OL{h} )}_1 \in \calL^1_G ( \Omega )$,
	which may be deduced from the proof of (ii) of \lref{l:sym}.
	Denoting $\bar{B}=T^{-\OL{h}} (B)$, we have 
	\begin{align}
		\log D^{( \OL{h} )}_1
		&=
		\int _0^1 \OL{h}_s \cdot d \bar{B}_s
		+
		\frac{1}{2} \int _0^1 \OL{h}_s \cdot 
		\left( 
		d \qv{B}_s \, \OL{h}_s
		\right) .
		\label{e:dot}
	\end{align}
	Taking the sublinear expectation of both sides 
	of \eqref{e:dot} under $\E^{ \OL{h} }$,
	we have by (ii) of \lref{l:sym},
	\begin{align}
		\E^{\OL{h}}
		\left[
		\log D^{( \OL{h} )}_1
		\right]
		&=
		\frac{1}{2} \E^{\OL{h}}
		\left[
		\int _0^1 \OL{h}_s \cdot 
		\left( 
		d \qv{B}_s \, \OL{h}_s
		\right)
		\right] \\
		&=
		\frac{1}{2} \E^{\OL{h}}
		\left[
		\int _0^1 h_s (\bar{B}) \cdot 
		\left( 
		d \qv{\bar{B}}_s \, h_s (\bar{B})
		\right)
		\right] ,
		\label{e:omr}
	\end{align}
	where the second equality follows from \eqref{;barh} 
	and the obvious identity $\qv{B}=\qv{\bar{B}}$.
	Then by \tref{;gir} and \rref{rem:gir}, 
	\begin{align}
		\frac{1}{2} \E ^{\OL{h}}
		\left[ 
		\int _0^1 h_s (\bar{B}) \cdot 
		\left( 
		d \qv{\bar{B}}_s \, h_s (\bar{B})
		\right) 
		\right] 
		&=
		\frac{1}{2} \E 
		\left[ 
		\int _0^1 h_s (B) \cdot 
		\left( 
		d \qv{B} _s \, h_s(B)
		\right) 
		\right] \\
		&\le 
		\frac{1}{2} \sigma _1^2
		\E 
		\left[ 
		\int _0^1 | h_s (B) |^2 \, d s
		\right] , 
	\end{align}
	and therefore the assertion is proved. 
\end{proof}

\begin{lem}\label{;key}
	Let $f : \Omega \to \R$ be bounded and Lipschitz continuous.
	Then, for every $h \in ( \MG )^d$, we have 
	$f \left( T^h (B) \right) \in \calL^1_G (\Omega )$. 
\end{lem}

\begin{proof}
	For each $n \in \N$, let the partition $\Delta _n$ of $[0,1]$ and 
	the polygonal approximation $(\omega )^{\Delta _{n}}$ 
	of $\omega \in \Omega$ as in the proof of \lref{lem:apl}. 
	Since $f$ is continuous and $\int _0^t d \qv{B} _s \, h_s$ 
	has a q.c.\ version for each $t \in [0,1]$, 
	it is clear that the functional 
	$f \left( \left( T^h (B) \right)^{ \Delta _n } \right) $ 
	also has a q.c.\ version, 
	hence belongs to $\calL^1_G ( \Omega )$ due to \eqref{;char1}.
	Therefore, in order to prove the lemma, it is sufficient to show 
	\begin{align}\label{;eqkey1}
		\ol{\E} 
		\left[ 
		\left| 
		f \left( T^h (B) \right)
		- f \left( \left( T^h (B) \right)^{ \Delta _n } \right) 
		\right| 
		\right] 
		\xrightarrow[ n \to \infty ]{} 0 .
	\end{align}
	To this end, note that, similarly to \eqref{eq:dist}, we have
	\begin{align}
		\rho \left( \omega , ( \omega )^{ \Delta_n } \right)
		&\le 
		2 \max _{1 \le k \le n}
		\sup_{t_{k-1} \le t \le t_{k}}
		\left| 
		\omega_t - \omega_{t_{k-1}}
		\right| 
		\qn{for } \omega \in \Omega .
	\end{align}
	Combining this estimate with Lipschitz continuity of $f$, we see that the 
	left-hand side of \eqref{;eqkey1} is bounded from above by 
	\begin{align}
		2 \lip (f) \left( I^1_n + I^2_n \right) 
	\end{align}
	with
	\begin{align}
		I^1_n
		&= 
		\ol{\E} 
		\left[
		\max_{1 \le k \le n} 
		\sup_{t_{k-1} \le t \le t_k}
		\left| B_t - B_{t_{k-1}} \right| 
		\right] , \\
		I^2_n
		&= 
		\ol{\E} 
		\left[
		\max_{1 \le k \le n} 
		\sup_{t_{k-1} \le t \le t_k}
		\left| \int_{t_{k-1}}^t d \qv{B}_s \, h_s \right| 
		\right] .
	\end{align}
	As to $I^2_n$,
	we see that by \eqref{;qv} and the Cauchy-Schwarz inequality,
	\begin{align}
		I_n^2
		&\le
		\sigma_1^2 \ol{\E}
		\left[
		\max_{1 \le k \le n} \sup_{t_{k-1} \le t \le t_k}
		\int _{t_{k-1}}^t | h_s | \, d s
		\right] \\
		&\le
		\sigma_1^2 \ol{\E}
		\left[
		\max_{1 \le k \le n} \sup_{t_{k-1} \le t \le t_k}
		\sqrt{t- t_{k-1}}
		\left(
		\int _0^1 | h_s |^2 \, d s
		\right) ^{1/2}
		\right] \\
		&\le 
		\sigma_1^2 n^{-1/2} \| h \|_{\MGR} .
	\end{align}
	Therefore $I^2_n$ tends to $0$ as $n \to \infty$.
	In order to see the convergence of $I^1_n$ to 0, fix $\ve >0$ arbitrarily.
	By the tightness of the family $\left\{ P_{\theta} : \theta \in \cA \right\}$,
	there exists a compact $K_{\ve} \subset \Omega$ 
	such that $c \left( K_{\ve}^c \right) \le \ve$.
	We bound $I^1_n$ from above by the sum
	\begin{align}\label{eq:jnb}
		\ol{\E} 
		\left[
		\max_{1 \le k \le n} 
		\sup_{t_{k-1} \le t \le t_k}
		\left| B_t - B_{t_{k-1}} \right|
		; K_{\ve}
		\right]
		+ 
		\ol{\E} 
		\left[
		\max_{1 \le k \le n} 
		\sup_{t_{k-1} \le t \le t_k}
		\left| B_t - B_{t_{k-1}} \right|
		; K_{\ve}^c	
		\right] .
	\end{align}
	Due to the uniform equicontinuity of $K_{\ve}$ by the Arzel\`a-Ascoli theorem 
	(see \cite[Theorem~7.2]{Billingsley;99}),
	the first term of \eqref{eq:jnb} converges to $0$ as $n \to \infty$.
	On the other hand,
	the second term of \eqref{eq:jnb} is dominated by
	\begin{align}
		2 \ol{\E} 
		\left[
		\sup_{0 \le t \le 1} \left| B_t \right|
		; K_{\ve}^c
		\right]
		\le 
		2 \ol{\E} 
		\left[
		\sup_{0 \le t \le 1} \left| B_t \right|^2
		\right] ^{1/2}
		c \left( K_{\ve}^c \right)^{1/2}
	\end{align}
	by the Cauchy-Schwarz inequality.
	Combining these with the estimate given in \eqref{C} yields
	\begin{align}
		\limsup_{n \to \infty} I^1_n
		\le 4 \sqrt{d} \sigma_1 \sqrt{\ve},
	\end{align}
	which leads to \eqref{;eqkey1} as $\ve$ is arbitrary.
\end{proof}

Using Lemmas~\ref{;appli} and \ref{;key}, we prove \pref{;abs}. 

\begin{proof}[Proof of \pref{;abs}]
	Let $\delta >0$ and $A \in \calB ( \Omega )$ be such that 
	$\hat{c} \left( A \right) < \delta$. 
	Let $h \in ( \MG )^d$ and fix a compact subset $K \subset A$ arbitrarily. 
	We take a sequence $\{ h^n \} _{n \in \N} \subset (\sbl )^d$ so that
	\begin{align}
		\lim_{n \to \infty}
		\| h - h^n \|_{ \MGR }
		= 0 . 
		\label{hn}
	\end{align}
	For each $m \in \N$, set the function $g_m$ by 
	$g_m (a) = 1- ma$ for $0 \le a \le1/m$ and 
	$g_m (a) =0$ for $a > 1/m$.
	Then the mapping $\Omega \ni \omega \mapsto g_m ( \rho ( \omega , K ) )$
	is bounded and Lipschitz continuous, and hence by \lref{;key},
	\begin{align}
		g_m \left( \rho \left( T^h (B) , K \right) \right) ,\ 
		g_m \left( \rho \left( T^{h^n} (B) , K \right) \right) 
		\in 
		\calL^1_G (\Omega ), 
		\qd m=1,2, \ldots .
		\label{gmL1G}
	\end{align}
	Fix $\alpha >1$. 
	We start with the proof of 
	\begin{align}\label{;estcap4}
		\E 
		\left[ 
		g_m \left( \rho \left( T^h (B) , K \right) \right) 
		\right] 
		\le 
		\alpha \E 
		\left[ 
		g_m \left( \rho \left( B , K \right) \right) 
		\right]
		+
		\frac{1}{\log \alpha} C_h , \\
		C_h
		:= 
		\frac{1}{e} 
		+ \frac{1}{2} \sigma _1^2 
		\sup _{n \in \N} 
		\| h^n \|_{\MGR}^2 ;
	\end{align}
	note that $C _h < \infty$ by \eqref{hn}. 
	For each $n\in \N $, 
	we bound the left-hand side of \eqref{;estcap4} from above by 
	\begin{align*}
		\E 
		\left[ 
		\left| 
		g_m \left( \rho \left( T^h (B) , K \right) \right) 
		-
		g_m \left( \rho \left( T^{h^n} (B) , K \right) \right) 
		\right| 
		\right]
		+
		\E 
		\left[ 
		g_m \left( \rho \left( T^{h^n} (B) , K \right) \right) 
		\right] . 
		\label{estgm}
	\end{align*}
	The first term is dominated by 
	\begin{align}
		\lip ( g_m ) \ol{\E}
		\left[
		\left|
		\sup _{0 \le t \le 1}
		\int_0^t d \qv{B} _s \, ( h_s - h^n_s )
		\right|
		\right]
		\le
		\lip ( g_m )
		\sigma _1^2 \| h - h^n \|_{ \MGR } ,
	\end{align}
	which vanishes by letting $n \to \infty $. 
	As to the second term of \eqref{estgm}, 
	we let $\OL{h^n} $ be a bounded element of $M^{2,0}_G (0,1)$ 
	that satisfies \eqref{;barh} with $h$ replaced by $h^n$. 
	Noting that \eqref{;barh} yields the relation 
	$T^{h^n} \circ T^{-\OL{h^n}} (B) = B$,
	we have, by \eqref{gmL1G} and \tref{;gir} (see \rref{rem:gir}), 
	\begin{align}
		\E 
		\left[ 
		g_m \left( \rho \left( T^{h^n} (B) , K \right) \right) 
		\right]
		=
		\E ^{\OL{h^n}}
		\left[ 
		g_m \left( \rho \left( B , K \right) \right) 
		\right] .
		\label{gm}
	\end{align}
	Following an argument in the proof of \cite[Lemma~2.8]{Boue;98},
	we estimate \eqref{gm} from above by
	\begin{align}
		&\ol{\E}
		\left[ 
		g_m \left( \rho \left( B , K \right) \right) D^{( \OL{h^n} )}_1 ;
		D^{( \OL{h^n} )}_1 \le \alpha
		\right]
		+
		\ol{\E}
		\left[ 
		g_m \left( \rho \left( B , K \right) \right) D^{( \OL{h^n} )}_1 ;
		D^{( \OL{h^n} )}_1 > \alpha
		\right] \\
		&\le
		\alpha
		\E 
		\left[ 
		g_m \left( \rho \left( B , K \right) \right) 
		\right]
		+
		\frac{1}{\log \alpha} 
		\ol{\E} 
		\left[
		g_m \left( \rho \left( B , K \right) \right) 
		D^{( \OL{h^n} )}_1 \log D^{( \OL{h^n} )}_1 ;
		D^{( \OL{h^n} )}_1 > \alpha
		\right] \\
		&\le
		\alpha
		\E 
		\left[ 
		g_m \left( \rho \left( B , K \right) \right) 
		\right]
		+
		\frac{1}{\log \alpha} 
		\left(
		\frac{1}{e}
		+
		\E 
		\left[
		D^{( \OL{h^n} )}_1 \log D^{( \OL{h^n} )}_1 
		\right]
		\right) .
		\label{estgm2}
	\end{align}
	Here for the last line, 
	we used the inequality 
	\begin{align}
		\one _{( \alpha , \infty )} (x) x \log x
		\le \frac{1}{e} + x \log x
		\qn{for all } x>0 .
	\end{align}
	Using \lref{;appli}, we see that \eqref{estgm2}
	is dominated from above by
	\begin{align}
		\alpha
		\E 
		\left[ 
		g_m \left( \rho \left( B , K \right) \right) 
		\right]
		+
		\frac{1}{\log \alpha} 
		C_h .
	\end{align}
	Combining these leads to \eqref{;estcap4}. 
	
	For the the left-hand side of \eqref{;estcap4}, we have 
	\begin{align}
		c \left( T^h (B) \in K \right) 
		\le
		\E \left[ g _m \left( \rho \left( T^h (B) , K \right) \right) \right]
		\label{estcapL}
	\end{align}
	since $\one _K ( \omega ) \le g _m ( \rho ( \omega , K) )$ 
	for all $\omega \in \Omega$. 
	On the other hand, 
	when we let $m \to \infty$, 
	$g _m \left( \rho \left( \omega , K \right) \right) $ converges to 
	$\one _K ( \omega )$ for all $\omega \in \Omega $ by the closedness of $K$, 
	and this convergence is decreasingly monotone. 
	Therefore by \cite[Theorem~31]{Denis;10} and \eqref{;char3}, 
	\begin{align}
		\lim_{m \to \infty}
		\E \left[ g_m \left( \rho \left( B,K \right) \right) \right]
		=
		\hat{\E} \left[ \one _K (B) \right]
		\le 
		\hat{c} \left( A \right) ,
		\label{estcapR}
	\end{align}
	where the inequality follows from the inclusion $K \subset A$.
	Then by \eqref{;estcap4}, \eqref{estcapL} 
	and \eqref{estcapR} with $\alpha = 1/ \sqrt{\delta}$, 
	we have for any compact subset $K \subset A$,
	\begin{align}
		c \left( T^h (B) \in K \right) 
		&\le
		\frac{1}{\sqrt{\delta}}
		\hat{c} \left( A \right)
		+
		\frac{2}{\log (1/ \delta )} C_h \\
		&< 
		\sqrt{\delta}
		+
		\frac{2}{\log (1/ \delta )} C_h
	\end{align}
	as $\hat{c} \left( A \right) < \delta$. 
	Therefore 
	\begin{align}
		c
		\left( 
		T^h (B) \in A
		\right) 
		&=
		\sup _{\theta \in \cA}
		P_{\theta} \circ \left( T^h (B) \right) ^{-1}
		(A) \\
		&=
		\sup _{\theta \in \cA}
		\sup _{
		\begin{subarray}{c}
			K \subset A \\
			K \n{: compact}
		\end{subarray}
		}
		P_{\theta} \circ \left( T^h (B) \right)^{-1}
		(K) \\
		&=
		\sup _{
		\begin{subarray}{c}
			K \subset A\\
			K \n{: compact}
		\end{subarray}
		}
		c \left( T^h (B) \in K \right) \\
		&\le
		\sqrt{\delta}
		+
		\frac{2}{\log (1/ \delta )} C_h ,
		\label{capth}
	\end{align}
	where the second line follows from the fact that 
	$P_{\theta} \circ \left( T^h (B) \right )^{-1}$ is a regular measure 
	for each $\theta \in \cA$ as $\Omega $ is a 
	complete separable metric space. 
	As the rightmost side of \eqref{capth} can be arbitrarily small 
	by letting $\delta \downarrow 0$, we conclude \eqref{abs1}. 
	Moreover, as mentioned just before \rref{nimed}, 
	$c \left( N \right)=0$ implies $\hat{c} \left( N \right) = 0$, 
	from which we have \eqref{abs2}. 
\end{proof}

%%%%%%%%%%%%%%%%%%%%%%%%%%%%%%%%%%%%%%%%%%%%%%%%%%%%%%%%%%%%%%%%%%%%%
%%%%%%%%%%%%%%%%%%%%%%%%%%%%%%%%%%%%%%%%%%%%%%%%%%%%%%%%%%%%%%%%%%%%%
%%%%%%%%%%%%%%%%%%%%%%%%%%%%%%%%%%%%%%%%%%%%%%%%%%%%%%%%%%%%%%%%%%%%%
%%%%%%%%%%%%%%%%%%%%%%%%%%%%%%%%%%%%%%%%%%%%%%%%%%%%%%%%%%%%%%%%%%%%%
%%%%%%%%%%%%%%%%%%%%%%%%%%%%%%%%%%%%%%%%%%%%%%%%%%%%%%%%%%%%%%%%%%%%%

\subsubsection*{Acknowledgements}
The author would like to express her sincere gratitude 
to Professor Yuu Hariya for invaluable suggestions and advice.
%His meticulous comments were an enormous help to her.
This work was partially supported by 
Research Fellowships of the Japan Society for the Promotion of 
Science for Young Scientists.

%%%%%%%%%%%%%%%%%%%%%%%%%%%%%%%%%%%%%%%%%%%%%%%%%%%%%%%%%%%%%%%%%%%%%
%%%%%%%%%%%%%%%%%%%%%%%%%%%%%%%%%%%%%%%%%%%%%%%%%%%%%%%%%%%%%%%%%%%%%
%%%%%%%%%%%%%%%%%%%%%%%%%%%%%%%%%%%%%%%%%%%%%%%%%%%%%%%%%%%%%%%%%%%%%
%%%%%%%%%%%%%%%%%%%%%%%%%%%%%%%%%%%%%%%%%%%%%%%%%%%%%%%%%%%%%%%%%%%%%
%%%%%%%%%%%%%%%%%%%%%%%%%%%%%%%%%%%%%%%%%%%%%%%%%%%%%%%%%%%%%%%%%%%%%

\end{document}